\documentclass{article}
\pdfoutput=1
\usepackage{amsmath,amsthm,amssymb,cases,amscd}
\usepackage{indentfirst}
\usepackage{ascmac}
\usepackage{yhmath}
\usepackage{mathrsfs}
\usepackage{textcomp}
\usepackage[toc,page]{appendix}
\usepackage{braket}
\usepackage{enumerate}
\usepackage{multienum}
\usepackage{mathtools}

\makeatletter
    
    \@addtoreset{equation}{section}
\makeatother

\usepackage[top=30truemm,bottom=30truemm,left=25truemm,right=25truemm]{geometry} 
\allowdisplaybreaks

\theoremstyle{plain}
\newtheorem{definition}{Definition}[section]
\newtheorem{lemma}{Lemma}[section]
\newtheorem{proposition}[lemma]{Proposition}
\newtheorem{theorem}[lemma]{Theorem}
\newtheorem{corollary}[lemma]{Corollary}

\theoremstyle{definition}

\newcommand{\lra}{\longrightarrow}

\newcommand{\into}{\hookrightarrow}
\newcommand{\onto}{\twoheadrightarrow}

\newcommand{\myim}{\operatorname{Im}}
\newcommand{\myre}{\operatorname{Re}}

\newcommand{\A}{\mathbb{A}}
\newcommand{\C}{\mathbb{C}}
\newcommand{\Q}{\mathbb{Q}}
\newcommand{\R}{\mathbb{R}}

\newcommand{\Z}{\mathbb{Z}}

\newcommand{\bfe}{\mathbf{e}}

\newcommand{\calD}{\mathcal{D}}

\newcommand{\calH}{\mathcal{H}}

\newcommand{\calO}{\mathcal{O}}

\newcommand{\itH}{\mathit{H}}
\newcommand{\itN}{\mathit{N}}

\newcommand{\frakH}{\mathfrak{H}}
\newcommand{\frakS}{\mathfrak{S}}

\newcommand{\an}[1]{\langle #1 \rangle}

\newcommand{\M}{\operatorname{M}}
\newcommand{\GL}{\operatorname{GL}}

\newcommand{\Sp}{\operatorname{Sp}}
\newcommand{\GSp}{\operatorname{GSp}}
\newcommand{\PGSp}{\operatorname{PGSp}}
\newcommand{\Mp}{\operatorname{Mp}}
\newcommand{\SO}{\operatorname{SO}}
\newcommand{\Or}{\operatorname{O}}
\newcommand{\SL}{\operatorname{SL}}
\newcommand{\mmp}{\mathcal{M}p}
\newcommand{\U}{\operatorname{U}}

\newcommand{\mmatrix}[4]{\begin{pmatrix} #1 & #2 \\ #3 & #4 \end{pmatrix}}
\newcommand{\tp}[1]{\prescript{\mathrm t}{}{#1}}

\newcommand{\tr}{\operatorname{tr}}

\newcommand{\SW}{\mathit{SW}}

\newcommand{\spin}{\mathrm{spin}}

\newcommand{\Irr}{\operatorname{Irr}}

\newcommand{\Hom}{\operatorname{Hom}}

\newcommand{\Sym}{\operatorname{Sym}}

\newcommand{\inv}{^{-1}}
\newcommand{\wtil}{\widetilde}
\newcommand{\what}{\widehat}

\title{Ibukiyama correspondences on automorphic forms on $\Mp_4(\A_\Q)$ and $\SO_5(\A_\Q)$ generating large discrete series representations at the real place}
\date{[\today]}
\author{Hiroshi Ishimoto}

\begin{document}

\maketitle

\begin{abstract}
  In our previous paper we gave proofs of Ibukiyama's correspondences on holomorphic Siegel modular forms of degree 2 of half-integral weight and integral weight.
  In this paper, we formulate and prove similar correspondences on automorphic forms on $\Mp_4(\A_\Q)$ or $\SO_5(\A)$ generating large discrete series representations at the real components.
  In addition, we show that the correspondences can be described in terms of local theta correspondences.
\end{abstract}

\setcounter{tocdepth}{1}
\tableofcontents

\section{Introduction} \label{Intr}
In 2008 and 2014, Ibukiyama \cite{ibuconj,iburef} proposed the following conjectures on vector valued holomorphic Siegel modular forms of degree 2 of half-integral weight and integral weight.
By using the Arthur classification for $\Mp_4$ (\cite{gi20}) and $\SO_5$ (\cite{art}), the conjectures have been proved (\cite{ishi}).
\begin{theorem}[Shimura type isomorphism on the Neben type]
  For any integer $k\geq3$ and any even integer $j\geq0$, there is a linear isomorphism
  \begin{equation*}
    S_{\det^{j+3}\Sym_{2k-6}}(\Sp_4(\Z)) \simeq S_{\det^{k-\frac{1}{2}}\Sym_j}^+(\Gamma_0(4), \left(\frac{-1}{\cdot}\right)),
  \end{equation*}
  which preserves $L$-functions.
\end{theorem}
Here "+" denotes the Kohnen plus space.
\begin{theorem}[Lifting to the Haupt type]
  For any integer $k\geq0$ and any even integer $j\geq0$, there exists an injective linear map
  \begin{equation*}
    \mathcal{L} : S_{2k-4}(\SL_2(\Z)) \otimes S_{2k+2j-2}(\SL_2(\Z)) \lra S_{\det^{k-\frac{1}{2}}\Sym_j}^+(\Gamma_0(4)),
  \end{equation*}
  such that if $f \in S_{2k-4}(\SL_2(\Z))$ and $g \in S_{2k+2j-2}(\SL_2(\Z))$ are Hecke eigenforms, so is $\mathcal{L}(f\otimes g)$, and they satisfy
  \begin{equation*}
    L(s, \mathcal{L}(f\otimes g)) = L(s-j-1, f) L(s,g).
  \end{equation*}
\end{theorem}
We shall write $S_{k-\frac{1}{2},j}^{+,0}(\Gamma_0(4))$ for the orthogonal complement of the image of the injective map $\mathcal{L}$.
\begin{theorem}[Shimura type isomorphism on the Haupt type]
  For any integer $k\geq3$ and any even integer $j\geq0$, there exists a linear isomorphism
  \begin{equation*}
    S_{\det^{j+3}\Sym_{2k-6}}(\Sp_4(\Z)) \simeq S_{\det^{k-\frac{1}{2}}\Sym_j}^{+,0}(\Gamma_0(4)),
  \end{equation*}
  which preserves $L$-functions.
\end{theorem}

In those theorems, $j$ is assumed to be even.
This parity condition is necessary, since the plus spaces $S_{\det^{k-\frac{1}{2}}\Sym_j}^+(\Gamma_0(4), \left(\frac{-1}{\cdot}\right)^l)$ ($l=0,1$) are identically zero while the spaces $S_{\det^{j+3}\Sym_{2k-6}}(\Sp_4(\Z))$ of integral weight remains nonzero in general otherwise.
Then, as Ibukiyama \cite{ibuconj} has pointed out, the following question arises: What kind of space should replace the plus space in the case $j$ is odd?
The first purpose of this paper is to give an answer to the question.

Tomoyoshi Ibukiyama suggested an expectation that the answer is a space of cuspidal automorphic forms on $\Mp_4(\A_\Q)$ generating large discrete series representations at the real place.
In this paper, we will give a formulation of the expectation and prove it.
We shall write as $\mathcal{A}_{\det^a\Sym_b}(\Mp_4)^{E_\psi}$ ($a\in\Z+\tfrac{1}{2}$, $b\in\Z$) for the space, which will be defined precisely in Subsection \ref{maindef}.
Then the following theorems are the first main results of this paper.
\begin{theorem}[Theorem \ref{main}]\label{a1}
  For any integer $k\geq3$ and any odd integer $j\geq1$, there exists a linear isomorphism
  \begin{equation*}
    S_{\det^{j+3}\Sym_{2k-6}}(\Sp_4(\Z)) \simeq \begin{cases*}
      \mathcal{A}_{\det^{-k+\frac{5}{2}}\Sym_{2k+j-3}}(\Mp_4)^{E_\psi},   & if $k$ is odd,  \\
      \mathcal{A}_{\det^{-k-j+\frac{1}{2}}\Sym_{2k+j-3}}(\Mp_4)^{E_\psi}, & if $k$ is even,
    \end{cases*}
  \end{equation*}
  which preserves $L$-functions.
\end{theorem}
\begin{theorem}[Theorem \ref{lifting}]\label{a2}
  For any integer $k\geq0$ and any odd integer $j\geq1$, there exists an injective linear map
  \begin{equation*}
    \mathscr{L} : S_{2k-4}(\SL_2(\Z)) \otimes S_{2k+2j-2}(\SL_2(\Z)) \lra \begin{cases*}
      \mathcal{A}_{\det^{-k-j+\frac{1}{2}}\Sym_{2k+j-3}}(\Mp_4)^{E_\psi}, & if $k$ is odd,  \\
      \mathcal{A}_{\det^{-k+\frac{5}{2}}\Sym_{2k+j-3}}(\Mp_4)^{E_\psi},   & if $k$ is even,
    \end{cases*}
  \end{equation*}
  such that if $f \in S_{2k-4}(\SL_2(\Z))$ and $g \in S_{2k+2j-2}(\SL_2(\Z))$ are Hecke eigenforms, then so is $\mathscr{L}(f\otimes g)$, and they satisfy
  \begin{equation*}
    L(s, \mathscr{L}(f\otimes g)) = L(s+k-\tfrac{5}{2}, f) L(s+j+k-\tfrac{3}{2},g).
  \end{equation*}
\end{theorem}
Let $(\myim\mathscr{L})^\perp$ denote the orthogonal complement of the image of the injective map $\mathscr{L}$.
\begin{theorem}[Theorem \ref{compl}]\label{a3}
  For any integer $k>3$ and any odd integer $j\geq1$, there exists a linear isomorphism
  \begin{equation*}
    S_{\det^{j+3}\Sym_{2k-6}}(\Sp_4(\Z)) \simeq (\myim\mathscr{L})^\perp,
  \end{equation*}
  which preserves $L$-functions.
  On the other hand, if $k=3$ and $j\geq1$ is odd, then there exists a linear isomorphism between $\mathcal{A}_{\det^{-j-\frac{5}{2}}\Sym_{j+3}}(\Mp_4)^{E_\psi}$ and the orthogonal complement of the image of the Saito-Kurokawa lifting in $S_{\det^{j+3}}(\Sp_4(\Z))$ which preserves $L$-functions.
\end{theorem}

Moreover, we can see that there is another similar isomorphism on a space of cuspidal automorphic forms on $\SO_5(\A_\Q)$ generating a large discrete series representation at the real place, in place of the space $S_{\det^{j+3}\Sym_{2k-6}}(\Sp_4(\Z))$ of holomorphic Siegel cusp forms of integral weight.
We shall write as $\mathcal{A}_{\det^a\Sym_b}(\SO_5)^{\mathrm{unr}}$ ($a, b\in\Z$) for the space, which will be defined precisely in Section \ref{maindef2}.
Then the following theorem is another main result of this paper.
\begin{theorem}[Theorem \ref{main2}]\label{b1}
  For any integers $k\geq3$ and $j\geq0$, there exists a linear isomorphism
  \begin{equation*}
    \mathcal{A}_{\det^{-j-1}\Sym_{2j+2k-2}}(\SO_5)^{\mathrm{unr}} \simeq \begin{cases*}
      \mathcal{A}_{\det^{-k+\frac{5}{2}}\Sym_{2k+j-3}}(\Mp_4)^{E_\psi},   & if $j$ is odd and $k$ is even,     \\
      \mathcal{A}_{\det^{-k-j+\frac{1}{2}}\Sym_{2k+j-3}}(\Mp_4)^{E_\psi}, & if $j$ is odd and $k$ is also odd, \\
      S_{\det^{k-\frac{1}{2}}\Sym_j}^+(\Gamma_0(4)),                      & if $j$ is even,
    \end{cases*}
  \end{equation*}
  which preserves $L$-functions.
\end{theorem}
We also have correspondences on $\mathcal{A}_{\det^{-j-1}\Sym_{2j+2k-2}}(\SO_5)^{\mathrm{unr}}$ similar to Theorems \ref{a2} and \ref{a3}.
We shall state them in Section \ref{mainthm2}.

In this paper we have an additional result Theorem \ref{absth}: The correspondences Theorems \ref{a1} (\ref{main}), \ref{a3} (\ref{compl}), and \ref{b1} (\ref{main2}) are given by the abstract theta correspondences.
The abstract theta lifting coincides with the automorphic theta lifting if the latter is nonzero and irreducible.
Thus the correspondences are constructed explicitly as automorphic theta lift if the theta lift is nonzero and irreducible.
Note that several equivalent (or sufficient) conditions to the non-vanishing and irreducibility of automorphic theta lifts are known.\\

This paper is organized as follows.
First, we recall the notion of modular forms of integral weight and half-integral weight in Section \ref{modform}.
Next, we recall representation theory of the metaplectic group $\Mp_{2n}$ of general degree in Subsections \ref{mp}-\ref{mpr}, and then we give the definition of $\mathcal{A}_{\det^a\Sym_b}(\Mp_4)^{E_\psi}$ in Section \ref{maindef}.
Then in Section \ref{mainthm}, main theorems are stated.
The family of these theorems is the first main result of this paper.
After that, we give the definition of $\mathcal{A}_{\det^a\Sym_b}(\SO_5)^{\mathrm{unr}}$ in Section \ref{maindef2}.
Then in Section \ref{mainthm2}, the second main theorem is stated.
In Section \ref{MF}, we review the multiplicity formulas for split odd special orthogonal group and metaplectic group of degree 2.
The multiplicity formulas play key roles.
We give proofs of the main theorems in Section \ref{pf}.
The idea of the proofs is similar to that of \cite{ishi}.
In Section \ref{relat} we describe theorems in Section \ref{mainthm} and \ref{mainthm2} in terms of the abstract theta lifting.
This is the third main result of this paper.

\subsection*{Convention \& Notation}
In this paper, for a number field $F$, we write $\A_F$ for the ring of adeles of $F$.
In particular, if $F$ is the field $\Q$ of rational numbers, let $\A_\Q$ be abbreviated as $\A$.
Moreover, let $\A_f$ be the kernel of the natural projection $\A\onto \R$.
For a nontrivial additive character $\psi$ of $F\backslash \A_F$ and $a\in F^\times$, put $\psi_a(x)=\psi(ax)$.
In particular when $a=-1$, we write $\overline{\psi}=\psi_{-1}$ and call it the complex conjugate of $\psi$.
Moreover, if $v$ is a place of $F$, we write $\psi_v$ for the local component of $\psi$ at $v$.

If $F$ is a local field of characteristic zero, let $W_F$ be the Weil group of $F$ and put
\begin{align*}
  L_F=
  \begin{cases*}
    W_F \times \SL_2(\C), & if $F$ is non-archimedean, \\
    W_F,                  & if $F$ is archimedean.
  \end{cases*}
\end{align*}
For any $a \in \frac{1}{2} \Z$, we shall write $\calD_a$ for the representation of $W_\R$ induced from the character $z=re^{i\theta} \mapsto e^{2ia\theta}$ of $W_\C=\C^\times$.
For a nontrivial additive character $\psi$ of $F$ and $a\in F^\times$, put $\psi_a(x)=\psi(ax)$.
For $z \in \C$, we set $\bfe(z)=\exp(2\pi i z)$.
We write $\mathcal{S}(F^n)$ for the Schwartz space on $F^n$.

For any representation $\pi$, let $\pi^\vee$ denote its contragredient representation.
For any finite abelian group $S$, let $\what{S}$ denote the group of characters of $S$.
For any positive integer $d\geq1$, we write $S_d$ for the unique irreducible representation of dimension $d$ of $\SL_2(\C)$.
In this paper, trivial characters are included in quadratic characters.

We shall write $\SO_{2n+1}$ for the split odd special orthogonal group of rank $n$, which is defined by
\begin{align*}
  \SO_{2n+1}
  =
  \Set{h \in \SL_{2n+1} | \tp{h} \left(\begin{array}{cc}1_{n+1}&\\&-1_n\end{array}\right) h = \left(\begin{array}{ccc}1_{n+1}&\\&-1_n\end{array}\right)}.
\end{align*}

\subsection*{Acknowledgment}
The author would like to thank Tomoyoshi Ibukiyama for telling the expectation that the substitute for the plus space of holomorphic Siegel cusp forms of half-integral weight of degree 2 is a space of automorphic forms for a large discrete series representation.
He also would like to thank Hiro-aki Narita for useful comments and encouragements.
This work is supported by JSPS KAKENHI Grant Number 22K20333 and JSPS Research Fellowships for Young Scientists KAKENHI Grant Number 23KJ1824.
The author also would like to appreciate Naoki Imai for his great support by JSPS KAKENHI Grant Number 22H00093.

\section{Review on modular forms}\label{modform}
In this section, we shall recall the definitions of holomorphic cusp forms and their $L$-functions.
For an integer $n\geq1$, let $\frakH_n$ denote the Siegel upper half space of degree $n$, i.e.,
\begin{equation*}
  \frakH_n
  =\Set{Z=X+iY \in \M_n(\C) | X=\tp{X}, Y=\tp{Y} \in \M_n(\R), \ Y >0 }.
\end{equation*}
Here, $Y>0$ means that $Y$ is positive definite.
We write $X=\myre(Z)$ and $Y=\myim(Z)$, for $Z=X+iY\in\frakH_n$, and call them the real part and the imaginary part, respectively.
By abuse of notation, we shall write $i$ for the scalar matrix in $\frakH_n$ of the diagonal entries $i$.
For any commutative ring $R$ with unity 1, let $\Sp_{2n}(R)$ be the symplectic group over $R$ of degree $n$, and $\GSp_{2n}(R)$ the symplectic similitude group over $R$ of degree $n$, i.e.,
\begin{align*}
  \Sp_{2n}(R)
   & =\Set{g \in \GL_{2n}(R) | \tp{g} J_n g = J_n},                              \\
  \GSp_{2n}(R)
   & =\Set{ g\in\GL_{2n}(R)| \tp{g}J_ng=\nu(g)J_n, \ \exists\nu(g)\in \GL_1(R)},
\end{align*}
where $J_n=\mmatrix{0}{1_n}{-1_n}{0}$.
The homomorphism $\nu : \GSp_{2n} \to \GL_1$ is called the similitude norm.
Note that $\Sp_2=\SL_2$ and $\GSp_2=\GL_2$.

The group $\GSp_{2n}^+(\R)=\{g \in \GSp_{2n}(\R) \mid \nu(g)>0\}$ acts on $\frakH_n$ in the following way:
\begin{align*}
  gZ & =(AZ+B)(CZ+D)\inv, & g & =\mmatrix{A}{B}{C}{D} \in \GSp_{2n}^+(\R), \quad Z\in\frakH_n.
\end{align*}
We set a factor of automorphy $J(g,Z)=CZ+D$.

\subsection{Elliptic cusp forms of integral weight}
We shall first review the notation and the definition of elliptic cusp forms and their $L$-functions.
Let $k$ be an integer.
\begin{definition}
  Let $f(z)$ be a $\C$-valued function on $z\in \frakH_1$.
  When $f$ satisfies the next conditions (0)-(2), we say that $f$ is an elliptic cusp form of weight $k$.
  \begin{enumerate}[(1)]
    \setcounter{enumi}{-1}
    \item $f$ is holomorphic;
    \item $J(\gamma, z)^{-k} f(\gamma z)=f(z)$ for all $\gamma \in \SL_2(\Z)$;
    \item $f$ has a Fourier expansion of the following form:
          \begin{align*}
            f(z)
            =\sum_{n\in\Z_{>0}} a_n \bfe(nz),
          \end{align*}
          where $\Z_{>0}$ denotes the set of all positive integers.
  \end{enumerate}
  We write $S_k(\SL_2(\Z))$ for the space of such functions.
\end{definition}

The Hecke operator $T(p)$ on $S_k(\SL_2(\Z))$ at a prime number $p$ is defined by
\begin{align*}
  [f|_k T(p)] (z)
  =p^{k-1} \sum_g J(g,z)^{-k} f(gz),
\end{align*}
where $g$ runs over a complete representative system of $\SL_2(\Z) \backslash \SL_2(\Z) \mmatrix{1}{0}{0}{p}\SL_2(\Z)$.

For a Hecke eigenform $f \in S_k(\SL_2(\Z))$, let $c_p$ denotes the eigenvalue of $T(p)$ associated with $f$.
Then the $L$-function of $f$ is defined as an Euler product
\begin{equation*}
  L(s,f)
  =\prod_p ( 1 - c_p p^{-s} + p^{k-1-2s} )\inv.
\end{equation*}

Finally, let us recall the Petersson inner product on $S_k(\SL_2(\Z))$.
For $f_1$, $f_2 \in S_k(\SL_2(\Z))$, it is defined by
\begin{equation*}
  \an{f_1, f_2}
  =\int_{\SL_2(\Z) \backslash \frakH_1} f_1(z) \overline{f_2(z)} y^{k-2} dx dy,
\end{equation*}
where $x$ and $y$ are the real and imaginary part of $z \in \mathfrak{H}_1$, respectively.

It is known that the Hecke operators are Hermite with respect to the Petersson inner product, and the $\C$-vector space $S_k(\SL_2(\Z))$ has an orthonormal basis consisting of Hecke eigenforms.
\subsection{Vector valued Siegel cusp forms of integral weight}
In this subsection we shall review the notation and the definition of vector valued Siegel cusp forms of degree 2 and their $L$-functions.
In order to do so, let us recall rational irreducible representations of $\GL_2(\C)$ here.
For any integer $j\geq0$, we write $(\Sym_j, V_j)$ for the symmetric tensor representation of degree $j$ (i.e., dimension $j+1$) of $\GL_2(\C)$.
Any rational irreducible representation of $\GL_2(\C)$ can be written as $\det^k \otimes \Sym_j$ by using an integer $k$ and a nonnegative integer $j$.
For any $V_j$-valued function $F(Z)$ on $\frakH_2$ and $g\in\GSp_4^+(\R)$, we define the slash operator of weight $\det^k\Sym_j$ by
\begin{align*}
  \left[F|_{k,j}g\right] (Z)
  =\nu(g)^{k+\frac{j}{2}} \det(J(g,Z))^{-k} \Sym_j(J(g,Z))\inv F(gZ).
\end{align*}

\begin{definition}
  Let $F(Z)$ be a $V_j$-valued function on $Z\in \frakH_2$.
  When $F$ satisfies the next conditions (0)-(2), we say that $F$ is a Siegel cusp form of weight $\det^k\Sym_j$.
  \begin{enumerate}[(1)]
    \setcounter{enumi}{-1}
    \item $F$ is holomorphic;
    \item $F|_{k,j} \gamma =F$ for all $\gamma \in \Sp_4(\Z)$;
    \item $F$ has a Fourier expansion of the following form:
          \begin{align*}
            F(Z)
            =\sum_{\substack{T \in L_2^* \\ T>0}} A(T) \bfe(\tr(TZ)),
          \end{align*}
          where $L_2^*$ denotes the set of all half-integral symmetric matrices of size $2\times2$.
  \end{enumerate}
  We write $S_{\det^k\Sym_j}(\Sp_4(\Z))$ for the space of such functions.
\end{definition}

Now we define Hecke operators $T(m)$ and the spinor $L$-functions.
For a positive integer $m$, put
\begin{align*}
  F|_{k,j} T(m)
  =m^{k+\frac{j}{2}-3} \sum_{g \in \Sp_4(\Z) \backslash X(m)} F|_{k,j}g,
\end{align*}
where
\begin{align*}
  X(m)
  =\set{ x \in \M_4(\Z) \cap \GSp_4^+(\R) | \nu(x)=m },
\end{align*}
and for a Hecke eigenform $F \in S_{\det^k\Sym_j}(\Sp_4(\Z))$, let $\lambda(m)$ denote the eigenvalue of $T(m)$ associated with $F$.
Note that $T(mn)=T(m)T(n)$ if $m, n \in \Z_{>0}$ are coprime.
The spinor $L$-function $L(s,F,\spin)$ of a Hecke eigenform $F\in S_{\det^k\Sym_j}(\Sp_4(\Z))$ is defined as an Euler product
\begin{equation*}
  L(s,F,\spin)
  =\prod_p L(s,F,\spin)_p,
\end{equation*}
where the product runs over all prime numbers $p$,
\begin{equation*}
  L(s,F,\spin)_p
  =\left(
  1 -\lambda(p)p^{-s} +(\lambda(p)^2-\lambda(p^2)-p^{\mu-1})p^{-2s} -\lambda(p)p^{\mu-3s} +p^{2\mu-4s}
  \right)\inv,
\end{equation*}
and $\mu=2k+j-3$.

As in the case of elliptic cusp forms, it is known that the Petersson inner product can be defined on $S_{\det^k\Sym_j}(\Sp_4(\Z))$, and the $\C$-vector space $S_{\det^k\Sym_j}(\Sp_4(\Z))$ has an orthonormal basis consisting of Hecke eigenforms.

\subsection{Vector valued Siegel cusp forms of half-integral weight}
Next, we shall review the notion of Siegel cusp forms of half-integral weight of degree 2, briefly.
Put
\begin{equation*}
  \Gamma_0(4)
  =\Set{\gamma=\mmatrix{A}{B}{C}{D} \in \Sp_4(\Z) | C\equiv0 \mod4},
\end{equation*}
and
\begin{equation*}
  \theta(Z)
  =\sum_{x\in\Z^2} \bfe(\tp{x} Z x),
  \quad Z\in\frakH_2.
\end{equation*}
Let us define a character $\left(\frac{-1}{\cdot}\right)$ of $\Gamma_0(4)$ by
\begin{align*}
  \left(\frac{-1}{\gamma}\right)
  =\left(\frac{-1}{\det D}\right),
  \quad
  \gamma
  =\left(\begin{array}{cc}
             A & B \\
             C & D
           \end{array}\right) \in \Gamma_0(4).
\end{align*}

In order to state the definition of vector valued Siegel modular forms of half-integral weight, we recall slash operators of half-integral weight.
We define the 4-fold covering group $\wtil{\GSp_4^+}(\R)$ of $\GSp_4^+(\R)$.
Let $\wtil{\GSp_4^+}(\R)$ be the set of pairs $(g, \phi(Z))$, where $g\in\GSp_4^+(\R)$ and $\phi(Z)$ is a $\C$-valued holomorphic function in $Z\in\frakH_2$ such that $\phi(Z)^4=\frac{\det J(g,Z)^2}{\det g}$.
It is a group under the multiplication law defined by
\begin{equation*}
  (g_1,\phi_1(Z)) (g_2, \phi_2(Z))
  =(g_1 g_2, \phi_1(g_2Z)\phi_2(Z)).
\end{equation*}
We can identify $\Gamma_0(4)$ with a subgroup $\wtil{\Gamma}_0(4)$ of $\wtil{\GSp_4^+}(\R)$ by embedding $\gamma \mapsto (\gamma, \theta(\gamma Z)/\theta(Z))$, and extend $\left(\frac{-1}{\cdot}\right)$ to $\wtil{\Gamma}_0(4)$ naturally.
Let $j\geq0$ be an integer.
For any $V_j$-valued function $F(Z)$ on $\frakH_2$ and $(g,\phi(Z))\in\wtil{\GSp_4^+}(\R)$, we define the slash operator of weight $\det^{k-\frac{1}{2}}\Sym_j$ by
\begin{align*}
  \left[F|_{k-\frac{1}{2},j} (g, \phi(Z)) \right] (Z)
  =\nu(g)^\frac{j}{2} \phi(Z)^{-2k+1} \Sym_j(J(g,Z))\inv F(gZ).
\end{align*}

\begin{definition}
  Let $F(Z)$ be a $V_j$-valued function on $Z\in \frakH_2$.
  When $F$ satisfies the next conditions (0)-(2), we say that $F$ is a Siegel cusp form of weight $\det^{k-\frac{1}{2}}\Sym_j$, level $\Gamma_0(4)$, character $\left(\frac{-1}{\cdot}\right)^l$ ($l\in(\Z/2\Z)$).
  \begin{enumerate}[(1)]
    \setcounter{enumi}{-1}
    \item $F$ is holomorphic;
    \item $F|_{k-\frac{1}{2},j} \wtil{\gamma} = \left(\frac{-1}{\wtil{\gamma}}\right)^lF$ for all $\wtil{\gamma} \in \wtil{\Gamma}_0(4)$;
    \item $F$ has a Fourier expansion of the following form:
          \begin{align*}
            F(Z)
            =\sum_{\substack{T \in L_2^* \\ T>0}} A(T) \bfe(\tr(TZ)),
          \end{align*}
  \end{enumerate}
  The space of such functions will be denoted by $S_{\det^{k-\frac{1}{2}}\Sym_j}(\Gamma_0(4), \left(\frac{-1}{\cdot}\right)^l)$.
  Moreover, we write $S_{\det^{k-\frac{1}{2}}\Sym_j}^+(\Gamma_0(4), \left(\frac{-1}{\cdot}\right)^l)$ for the subspace consisting of those $F\in S_{\det^{k-\frac{1}{2}}\Sym_j}(\Gamma_0(4), \left(\frac{-1}{\cdot}\right)^l)$ such that $A(T)=0$ unless $T\equiv(-1)^{k+l-1} r \tp{r} \mod4L_2^*$ for some $r \in \Z^2$.
  The subspace $S_{\det^{k-\frac{1}{2}}\Sym_j}^+(\Gamma_0(4), \left(\frac{-1}{\cdot}\right)^l)$ is called the plus space.
\end{definition}

As in \cite{ibuconj}, we define Hecke operators, and then $L$-function $L(s,F)$ for a Hecke eigenform $F$.
\section{Representations of metaplectic groups}
In this section we shall review the definition of metaplectic groups and some theories of representations of them.
After that, we define a space of the cuspidal automorphic forms on the metaplectic group of degree $2$ that generate large discrete series representations at the real place.
\subsection{Metaplectic groups}\label{mp}
We begin by recalling definitions and fixing notation on metaplectic groups.
Let first $F$ be a local field of characteristic 0.
The 2-cocycle $c_F(-,-):\Sp_{2n}(F)\times \Sp_{2n}(F) \to \{\pm1\}$ given by Ranga Rao \cite[\S5]{rao} defines a double cover
\begin{equation*}
  1 \lra \set{\pm1} \lra \Mp_{2n}(F) \lra \Sp_{2n}(F) \lra 1,
\end{equation*}
which is nonlinear unless $F=\C$.
The group $\Mp_{2n}(F)$ is called the metaplectic group over $F$.
We shall identify $\Mp_{2n}(F) = \Sp_{2n}(F) \times \{\pm1\}$ as sets, and then the multiplication law is given by
\begin{equation*}
  (g,\epsilon) \cdot (g',\epsilon')
  =(gg', \epsilon\epsilon' c_F(g,g')).
\end{equation*}
Moreover, if $F$ is a non-archimedean local field of residual characteristic other than 2, with the ring of integers $\calO$, then the covering splits over $\Sp_{2n}(\calO)$.
Let us write $s_F$ for this splitting:
\begin{equation}\label{splsph}
  \Sp_{2n}(\calO) \into \Mp_{2n}(F) , \quad g \mapsto (g, s_F(g)).
\end{equation}

Next, let $F$ be a number field.
For every place $v$ of $F$, we shall write $c_v=c_{F_v}$ and $s_v=s_{F_v}$, for simplicity.
For a finite set $\mathfrak{S}$ of places of $F$ that contains all places above $\infty$ and 2, put
\begin{equation*}
  \Sp_{2n}(\A_F)_\mathfrak{S}
  =\prod_{v\in\mathfrak{S}}\Sp_{2n}(F_v) \times \prod_{v \notin \mathfrak{S}}\Sp_{2n}(\mathcal{O}_v),
\end{equation*}
where $\mathcal{O}_v$ denotes the ring of integers of $F_v$.
Then we define the double covering $\Mp_{2n}(\A_F)_{\mathfrak{S}}\to\Sp_{2n}(\A_F)_{\mathfrak{S}}$ by the 2-cocycle $\prod_{v\in\mathfrak{S}}c_v(-,-)$.
In other words, we set
\begin{equation*}
  \Mp_{2n}(\A_F)_\mathfrak{S}
  =\Sp_{2n}(\A_F)_\mathfrak{S} \times \set{\pm1}
\end{equation*}
as sets, and define the multiplication by
\begin{equation*}
  ((g_v)_v, \epsilon)\cdot((g_v')_v, \epsilon')
  =((g_vg_v')_v, \epsilon\epsilon'\prod_{v \in \mathfrak{S}}c_v(g_v,g_v')).
\end{equation*}
For two such finite sets $\mathfrak{S}_1\subset\mathfrak{S}_2$, we define the embedding $\Mp_{2n}(\A_F)_{\mathfrak{S}_1}\into\Mp_{2n}(\A_F)_{\mathfrak{S}_2}$ by
\begin{equation*}
  ((g_v)_v, \epsilon)
  \mapsto ((g_v)_v, \epsilon\prod_{v \in \mathfrak{S}_2\setminus\mathfrak{S}_1}s_v(g_v)).
\end{equation*}
The global metaplectic group $\Mp_{2n}(\A_F)$ is defined by the inductive limit
\begin{equation*}
  \Mp_{2n}(\A_F)
  =\varinjlim_{\mathfrak{S}} \Mp_{2n}(\A_F)_\mathfrak{S},
\end{equation*}
where $\mathfrak{S}$ extends over all such finite sets of places of $F$.
This is a double cover of $\Sp_{2n}(\A_F)$:
\begin{equation*}
  1 \lra \set{\pm1} \lra \Mp_{2n}(\A_F) \lra \Sp_{2n}(\A_F) \lra 1,
\end{equation*}
but note that in the expression $(g,\epsilon)\in\Mp_{2n}(\A_F)_\frakS$, the second component $\epsilon=\pm1$ depends on the choice of $\mathfrak{S}$.
Thus, we do not identify $\Mp_{2n}(\A_F)$ with $\Sp_{2n}(\A_F)\times\{\pm1\}$ as sets.
The covering splits over $\Sp_{2n}(F)$, and the image of $\gamma \in \Sp_{2n}(F)$ is given by $(\gamma, 1) \in \Mp_{2n}(\A_F)_\mathfrak{S}$, for sufficiently large $\mathfrak{S}$ depending on $\gamma$.
In this paper, we regard $\Sp_{2n}(F)$ as a subgroup of $\Mp_{2n}(\A_F)$ in this way.\\

A representation of the metaplectic group over a local field or a ring of adeles is said to be genuine if it does not factor through the covering map.
A function on the metaplectic group is also said to be genuine if it satisfies the same condition.

\subsection{Unramified representations of metaplectic groups over $p$-adic fields}
In this subsection we treat the case that $F$ is a non-archimedean local field.
Let $\psi$ be a nontrivial additive character of $F$ of order zero.
We shall say that an irreducible genuine representation of $\Mp_{2n}(F)$ is $\psi$-unramified, or simply unramified, if the corresponding $L$-parameter via the LLC \cite{gs} with respect to $\psi$ is unramified.

If $p\neq2$, we shall regard $\Sp_{2n}(\calO)$ as a subgroup of $\Mp_{2n}(F)$ via the splitting \eqref{splsph}, and call a representation of $\Mp_{2n}(F)$ to be $\Sp_{2n}(\calO)$-spherical if it contains a nonzero vector fixed by $\Sp_{2n}(\calO)$.
Since $\psi$ has order zero, when $p\neq2$, it is known that an irreducible genuine representation of $\Mp_{2n}(F)$ is unramified if and only if it is $\Sp_{2n}(\calO)$-spherical.
When $p=2$, however, there is no notion of spherical representations of $\Mp_{2n}(F)$.
Nevertheless, regardless of the parity of $p$, by \cite[Theorems 5.4, 5.5]{ishi} and the higher degree generalization of \cite[Lemma 5.2, Proposition 5.3]{hi}, one can characterize unramified representations as follows.
Let $(\omega_\psi, \mathcal{S}(F^n))$ be the Schr\"{o}dinger model of the Weil representation of $\Mp_{2n}(F)$, $\Phi_0\in\mathcal{S}(F^n)$ the characteristic function of a lattice $\mathcal{O}^n$, and $dx$ the self-dual Haar measure on $F^n$.
Namely, the volume of $\mathcal{O}^n$ is 1 since the order of $\psi$ is 0.
Put
\begin{align*}
  \Gamma'= \Set{\mmatrix{A}{B}{C}{D}\in \Sp_{2n}(F) | A,D\in \M_2(\mathcal{O}),\ B\in\M_2(\tfrac{1}{4}\mathcal{O}),\ C\in\M_2(4\mathcal{O})},
\end{align*}
and
\begin{align*}
  E_\psi(g)=\begin{dcases*}
              |2|_F^{-n} \int_{F^n} \Phi_0(x) \overline{[\omega_{\psi}(g) \Phi_0](x)} dx, & if $g\in \wtil{\Gamma'}$, \\
              0,                                                                          & otherwise,
            \end{dcases*}
\end{align*}
where $\widetilde{\Gamma'}$ denotes the preimage of $\Gamma'$ under the covering map $\Mp_{2n}(F)\to \Sp_{2n}(F)$.

It is known that the restriction of the Weil representation $\omega_\psi$ of $\Mp_{2n}(F)$ to the metaplectic covering $\wtil{\Gamma}_0(4\mathcal{O})$ over
\begin{align*}
  \Gamma_0(4\mathcal{O})
  =\Set{
    \mmatrix{A}{B}{C}{D} \in \Sp_{2n}(\mathcal{O})
    |
    C \in 4\mathcal{O}
  }
\end{align*}
defines a genuine character $\varepsilon_\psi$ of $\wtil{\Gamma}_0(4\mathcal{O})$ by
\begin{align*}
  \omega_\psi(\gamma) \Phi_0
         & =\varepsilon_\psi(\gamma)\inv \Phi_0, &
  \gamma & \in\wtil{\Gamma}_0(4\mathcal{O}).
\end{align*}
Note that if $p\neq2$, $\varepsilon_\psi$ is quadratic and defines the splitting \eqref{splsph} over $\Gamma_0(4\mathcal{O})=\Sp_{2n}(\mathcal{O})$.
Let us define the Hecke algebra $\calH(\wtil{\Gamma}_0(4\mathcal{O})\backslash\Mp_{2n}(F)/\wtil{\Gamma}_0(4\mathcal{O});\varepsilon_\psi)$ to be the $\C$-algebra consisting of compactly supported functions $f:\Mp_{2n}(F) \to \C$ satisfying $f(\gamma_1g\gamma_2)=\varepsilon_\psi(\gamma_1 \gamma_2)f(g)$ for $\gamma_1, \gamma_2\in\widetilde{\Gamma}_0(4\mathcal{O})$ and $g\in\Mp_{2n}(F)$, equipped with the convolution
\begin{align*}
  [f_1 * f_2](g) & = \int_{\Sp_{2n}(F)} f_1(g {\tilde{x}\inv}) f_2(\tilde{x}) dx, &
  f_1, f_2 \in \calH(\wtil{\Gamma}_0(4\mathcal{O})\backslash\Mp_{2n}(F)/\wtil{\Gamma}_0(4\mathcal{O});\varepsilon_\psi),
\end{align*}
where $\tilde{x}$ denotes any lift of $x$.
It acts on a genuine smooth admissible representation $(\pi,V)$ of $\Mp_{2n}(F)$ by
\begin{align*}
  \pi(f)v & =\int_{\Sp_{2n}(F)} f(\widetilde{x})\pi(\widetilde{x})v dx,                                                           &
  f       & \in\calH(\wtil{\Gamma}_0(4\mathcal{O})\backslash\Mp_{2n}(F)/\wtil{\Gamma}_0(4\mathcal{O});\varepsilon_\psi),\ v\in V.
\end{align*}
Then $E_\psi\in\calH(\wtil{\Gamma}_0(4\mathcal{O})\backslash\Mp_{2n}(F)/\wtil{\Gamma}_0(4\mathcal{O});\varepsilon_\psi)$ is an idempotent, and a representation of $\Mp_{2n}(F)$ is said to be $\psi$-pseudospherical if it contains a nonzero $E_\psi$-stable vector.

We have the following proposition, which is a generalization of \cite[Proposition 5.3]{hi}.
\begin{proposition}\label{urm}
  Let $(\pi,V)$ be a irreducible genuine representation of $\Mp_{2n}(F)$.
  Then the followings are equivalent:
  \begin{enumerate}
    \item $(\pi,V)$ is $\psi$-pseudospherical, i.e., $V^{E_\psi}=\pi(E_\psi)V\neq0$;
    \item $\dim_\C V^{E_\psi}=1$;
    \item $(\pi,V)$ is $\psi$-unramified.
  \end{enumerate}
\end{proposition}
\begin{proof}
  Let us write $\Sp^J_{2n}=\Sp_{2n} \ltimes \mathcal{H}_n$ for the Jacobi group of degree $n$, where $\mathcal{H}_n$ denotes the Heisenberg group of degree $n$.
  Let $(\pi_{\SW,\psi}, \mathcal{S}(F^n))$ be the Schr\"{o}dinger-Weil representation associated to $\psi$ of $\Mp^J_{2n}(F)=\Mp_{2n}(F) \ltimes \mathcal{H}_n(F)$, the metaplectic double cover of $\Sp^J_{2n}(F)$.
  Put
  \begin{align*}
    \mathcal{S}^0 & =\mathcal{S}(2\inv \mathcal{O}^n/\mathcal{O}^n)                                                                                \\
                  & =\Set{ f\in\mathcal{S}(F^n) | \operatorname{supp}(f) \subset 2\inv \mathcal{O}^n,\ f(x+y)=f(x),\ \forall y\in \mathcal{O}^n }.
  \end{align*}
  Then by \cite[Propositions 2.1, 2.2]{su}, $\mathcal{S}^0$ is invariant with respect to the restriction $\omega_\psi^0=\omega_\psi|_{\widetilde{\Sp_{2n}(\mathcal{O})}}$, and the representation $(\omega_\psi^0, \mathcal{S}^0)$ of $\widetilde{\Sp_{2n}(\mathcal{O})}$ is irreducible, where $\widetilde{\Sp_{2n}(\mathcal{O})}$ denotes the preimage of $\Sp_{2n}(\mathcal{O})$ in $\Mp_{2n}(F)$.
  Then as \cite[Lemma 5.2]{hi}, we have
  \begin{align}\label{med}
    \Hom_{\widetilde{\Sp_{2n}(\mathcal{O})}} \left( \omega_\psi^0, \pi \right)
    =\Hom_{\Sp_{2n}(\mathcal{O})} \left( 1, \pi \otimes \omega_{\overline{\psi}}^0 \right)
    =\left( \pi \otimes \omega_{\overline{\psi}}^0 \right)^{\Sp_{2n}(\mathcal{O})}
    =\left( \pi \otimes \pi_{\SW, \overline{\psi}} \right)^{\Sp^J_{2n}(\mathcal{O})}.
  \end{align}
  The last equation holds because the subspace of $\mathcal{H}_n(\mathcal{O})$-fixed vectors of $(\pi_{\SW,\overline{\psi}}, \mathcal{S}(F^n))$ is exactly $\mathcal{S}^0$.

  Put
  \begin{align*}
    e_\psi(g)=E_\psi(w\inv g w),
  \end{align*}
  where $w\in\Mp_{2n}(F)$ is a lift of
  \begin{align*}
    \left( \begin{array}{cccccc}
                 &        &   & -\frac{1}{2} &        &              \\
                 &        &   &              & \ddots &              \\
                 &        &   &              &        & -\frac{1}{2} \\
               2 &        &   &              &        &              \\
                 & \ddots &   &              &        &              \\
                 &        & 2 &              &        &
             \end{array} \right)
    \in\Sp_{2n}(F).
  \end{align*}
  Then by the same argument as in the proof of \cite[Proposition 5.3]{hi}, we have
  \begin{align}\label{aka}
    \dim_\C V^{E_\psi}
    = \dim_\C V^{e_\psi}
    = \dim_\C \Hom_{\widetilde{\Sp_{2n}(\mathcal{O})}} \left( \omega_\psi^0, \pi \right).
  \end{align}

  The equations \eqref{med} and \eqref{aka} imply that $(\pi,V)$ is $\psi$-pseudospherical if and only if $\pi \otimes \pi_{\SW, \overline{\psi}}$ is $\Sp^J_{2n}(\mathcal{O})$-spherical.
  By \cite[Theorems 5.2, 5.4, 5.5]{ishi}, this is equivalent to the condition that $(\pi,V)$ is $\psi$-unramified.
  Hence the first and the third conditions are equivalent, and the second one implies them.

  Moreover, as remarked in the proof of \cite[Theorem 5.3]{ishi}, if $\pi \otimes \pi_{\SW, \overline{\psi}}$ is $\Sp^J_{2n}(\mathcal{O})$-spherical then the $\Sp^J_{2n}(\mathcal{O})$-invariant subspace is one dimensional.
  Hence the first condition implies the second one.
  This completes the proof.
\end{proof}

\subsection{Tempered representations of the metaplectic group over the real number field}\label{mpr}
Let next $F$ be the field $\R$ of real numbers and $\psi$ a nontrivial additive character of $\R$.
Let us recall from \cite[Theorem 8.1]{vog} that the irreducible genuine tempered representations of $\Mp_{2n}(\R)$ with real infinitesimal character are determined by the lowest $\widetilde{K_\infty}$-type.

The standard maximal compact subgroup $K_\infty$ of $\Sp_{2n}(\R)$ is
\begin{align*}
  K_\infty
  =\set{
    \left(\begin{array}{cc} \alpha&\beta\\\-\beta&\alpha\end{array}\right) \in \GL_{2n}(\R)
    |
    \tp{\alpha}\alpha+\tp{\beta}\beta=1_2, \ \tp{\alpha}\beta=\tp{\beta}\alpha
  }.
\end{align*}
Let $\widetilde{K_\infty}$ denote the preimage of $K_\infty$ under the covering map, and $\wtil{\U(n)}$ the $\det^{\frac{1}{2}}$ cover of the unitary group $\U(n)$ defined in \cite[p.7]{ada}.
Following \cite[p.19]{ada}, fix an identification $K_\infty \cong \U(n)$ (and $\wtil{K_\infty} \cong \wtil{\U(n)}$) by
\begin{align*}
  \left(\begin{array}{cc} \alpha&\beta\\\-\beta&\alpha\end{array}\right)
  \mapsto
  \begin{cases}
    \alpha + i\beta, & c_\psi>0, \\
    \alpha - i\beta, & c_\psi<0,
  \end{cases}
\end{align*}
where $c_\psi$ is a real number such that $\psi(x)=\mathbf{e}(c_\psi x)$.
Then irreducible genuine tempered representations of $\Mp_{2n}(\R)$ with real infinitesimal character are determined by the lowest $\widetilde{K_\infty}$-type, and we shall write $\pi_\Lambda$ for the irreducible genuine tempered representation of $\Mp_{2n}(\R)$ with lowest $\widetilde{K_\infty}$-type $\Lambda \in (\frac{1}{2}+\Z)^n$.

\subsection{Cuspidal automorphic representations of the metaplectic group with large discrete series representation at the real place}\label{maindef}
In this subsection, we fix $\psi$ to be the nontrivial additive character of $\Q \backslash \A$ such that $\psi_\infty=\bfe$.
Let $L^2(\Mp_4)$ be the subspace of $L^2(\Sp_4(F) \backslash \Mp_4(\A_F))$ consisting of the genuine functions, that is to say, the maximum subspace on which $\{\pm1\}$ acts as the nontrivial character.
Let us write $L^2_{\mathrm{disc}}(\Mp_4)$ for the discrete spectrum of the genuine unitary representation $L^2(\Mp_4)$ of $\Mp_4(\A_F)$.

Put
\begin{align*}
  \mathcal{H}(\wtil{\Gamma}_0(4)\backslash\Mp_4(\A_f)/\wtil{\Gamma}_0(4);\varepsilon_\psi)
  ={\bigotimes_{p<\infty}}' \mathcal{H}(\wtil{\Gamma}_0(4\Z_p)\backslash\Mp_4(\Q_p)/\wtil{\Gamma}_0(4\Z_p);\varepsilon_{\psi_p})
\end{align*}
to be the restricted tensor product with respect to $\{\varepsilon_{\psi_p}\}_p$.
Here we set $\varepsilon_{\psi_p}$ to be zero outside $\widetilde{\Gamma}_0(4\Z_p)$.
Note that $\varepsilon_{\psi_p}=E_{\psi_p}$ for $p\neq2$.
The Hecke algebra acts on $L^2_{\mathrm{disc}}(\Mp_4)$ by
\begin{align*}
  [f\cdot\varphi](g) = \int_{\Sp_4(\A)} f(\widetilde{x})\varphi(g\widetilde{x})dx,
\end{align*}
for $f \in \mathcal{H}(\wtil{\Gamma}_0(4)\backslash\Mp_4(\A_f)/\wtil{\Gamma}_0(4);\varepsilon_\psi)$, $\varphi \in L^2_{\mathrm{disc}}(\Mp_4)$, and $g \in \Mp_4(\A)$.
\begin{definition}
  Let $k\leq0$ and $j\geq0$ be integers such that $k+j>0$ and $2k+j-1\neq0,\pm1$.
  We define
  \begin{equation*}
    \mathcal{A}_{\det^{k-\frac{1}{2}}\Sym_j}(\Mp_4)^{E_\psi} = \Hom_{\Mp_4(\R)}(\pi_{(k+j-\frac{1}{2},k-\frac{1}{2})}, L^2_{\mathrm{disc}}(\Mp_4)^{E_\psi}),
  \end{equation*}
  where $\pi_{(k+j-\frac{1}{2},k-\frac{1}{2})}$ is the discrete series representation of $\Mp_4(\R)$ with lowest $\widetilde{K_\infty}$-type $(k+j-\frac{1}{2},k-\frac{1}{2})$ (relative to $\psi_\infty$), and $L^2_{\mathrm{disc}}(\Mp_4)^{E_\psi}$ denotes the subspace fixed by $E_\psi$, that is,
  \begin{align*}
    L^2_{\mathrm{disc}}(\Mp_4)^{E_\psi}=\Set{\varphi\in L^2_{\mathrm{disc}}(\Mp_4) | E_\psi\cdot\varphi=\varphi}.
  \end{align*}
\end{definition}

For a vector $\Phi\in\mathcal{A}_{\det^{k-\frac{1}{2}}\Sym_j}(\Mp_4)^{E_\psi}$, let $\pi_\Phi$ denote the automorphic representation of $\Mp_4(\A)$ generated by the image of $\Phi$.
Note that $L^2_{\mathrm{disc}}(\Mp_4)$ can be decomposed into a direct sum of irreducible automorphic representations $\pi$, and recall that the finite place component of $\pi^{E_\psi}$ is at most 1-dimensional by Proposition \ref{urm}.
Thus the $\C$-vector space $\mathcal{A}_{\det^{k-\frac{1}{2}}\Sym_j}(\Mp_4)^{E_\psi}$ can be decomposed into a direct sum of 1-dimensional subspaces which are spanned by vectors $\Phi$ such that $\pi_\Phi$ are irreducible automorphic representations.
By the definition of $\mathcal{A}_{\det^{k-\frac{1}{2}}\Sym_j}(\Mp_4)^{E_\psi}$, we have $\pi_{\Phi,\infty}=\pi_{(k+j-\frac{1}{2},k-\frac{1}{2})}$.
In this paper we shall call such $\Phi$ a Hecke eigenform.
Then we define the $L$-function $L(s,\Phi)$ of a Hecke eigenform $\Phi$ as
\begin{equation*}
  L(s,\Phi)
  =L(s,\pi_\Phi,\psi).
\end{equation*}
\section{Main theorems on cusp forms on the metaplectic group for large discrete series representations}\label{mainthm}
Now we can state our main theorems on cuspidal automorphic forms on $\Mp_4(\A)$ generating a large discrete series representation at the real place.
\begin{theorem}\label{main}
  For any integer $k\geq3$ and any odd integer $j\geq1$, there exists a linear isomorphism
  \begin{equation*}
    \rho^\itN : S_{\det^{j+3}\Sym_{2k-6}}(\Sp_4(\Z)) \overset{\simeq}{\lra} \begin{cases*}
      \mathcal{A}_{\det^{-k+\frac{5}{2}}\Sym_{2k+j-3}}(\Mp_4)^{E_\psi},   & if $k$ is odd,  \\
      \mathcal{A}_{\det^{-k-j+\frac{1}{2}}\Sym_{2k+j-3}}(\Mp_4)^{E_\psi}, & if $k$ is even,
    \end{cases*}
  \end{equation*}
  such that if $F \in S_{\det^{j+3}\Sym_{2k-6}}(\Sp_4(\Z))$ is a Hecke eigenform, then so is $\rho^\itN(F)$, and they satisfy
  \begin{equation*}
    L(s,\rho^\itN(F)) = L(s+j+k-\tfrac{3}{2},F,\spin).
  \end{equation*}
\end{theorem}

\begin{theorem}\label{lifting}
  For any integer $k\geq0$ and any odd integer $j\geq1$, there exists an injective linear map
  \begin{equation*}
    \mathscr{L} : S_{2k-4}(\SL_2(\Z)) \otimes S_{2k+2j-2}(\SL_2(\Z)) \lra \begin{cases*}
      \mathcal{A}_{\det^{-k-j+\frac{1}{2}}\Sym_{2k+j-3}}(\Mp_4)^{E_\psi}, & if $k$ is odd,  \\
      \mathcal{A}_{\det^{-k+\frac{5}{2}}\Sym_{2k+j-3}}(\Mp_4)^{E_\psi},   & if $k$ is even,
    \end{cases*}
  \end{equation*}
  such that if $f \in S_{2k-4}(\SL_2(\Z))$ and $g \in S_{2k+2j-2}(\SL_2(\Z))$ are Hecke eigenforms, then so is $\mathscr{L}(f\otimes g)$, and they satisfy
  \begin{equation*}
    L(s, \mathscr{L}(f\otimes g)) = L(s+k-\tfrac{5}{2}, f) L(s+j+k-\tfrac{3}{2},g).
  \end{equation*}
\end{theorem}

Note that if $k\leq7$, then $2k-4\leq10$, and hence $S_{2k-4}(\SL_2(\Z))=0$.
Thus the lifting theorem holds trivially with $\mathscr{L}=0$ when $k\leq7$.

Let $(\myim\mathscr{L})^\perp$ denote the orthogonal complement of the image of the injective map $\mathscr{L}$.
\begin{theorem}\label{compl}
  For any integer $k>3$ and any odd integer $j\geq1$, there exists a linear isomorphism
  \begin{equation*}
    \rho^\itH : S_{\det^{j+3}\Sym_{2k-6}}(\Sp_4(\Z)) \overset{\simeq}{\lra} (\myim\mathscr{L})^\perp,
  \end{equation*}
  such that if $F \in S_{\det^{j+3}\Sym_{2k-6}}(\Sp_4(\Z))$ is a Hecke eigenform, then so is $\rho^\itH(F)$, and they satisfy
  \begin{equation*}
    L(s,\rho^\itH(F)) = L(s+j+k-\tfrac{3}{2},F,\spin).
  \end{equation*}
  On the other hand, consider the case $k=3$ and $j\geq1$ is odd. Then there exists a linear isomorphism between $\mathcal{A}_{\det^{-j-\frac{5}{2}}\Sym_{j+3}}(\Mp_4)^{E_\psi}$ and the orthogonal complement of the Saito-Kurokawa lifting in $S_{\det^{j+3}}(\Sp_4(\Z))$ which preserves $L$-functions.
\end{theorem}

By these theorems, we obtain the following corollary.
\begin{corollary}
  For any integer $k>3$ and any odd integer $j\geq1$, we have
  \begin{align*}
     & \dim \mathcal{A}_{\det^{-k+\frac{5}{2}}\Sym_{2k+j-3}}(\Mp_4)^{E_\psi}                                                                              \\
     & \quad=\dim \mathcal{A}_{\det^{-k-j+\frac{1}{2}}\Sym_{2k+j-3}}(\Mp_4)^{E_\psi} + (-1)^k\dim S_{2k-4}(\SL_2(\Z)) \times \dim S_{2k+2j-2}(\SL_2(\Z)).
  \end{align*}
  When $k=3$ and $j\geq1$ is odd, we have
  \begin{equation*}
    \dim \mathcal{A}_{\det^{-j-\frac{5}{2}}\Sym_{j+3}}(\Mp_4)^{E_\psi} + \dim S_{2j+4}(\SL_2(\Z)) = \dim \mathcal{A}_{\det^{-\frac{1}{2}}\Sym_{j+3}}(\Mp_4)^{E_\psi}.
  \end{equation*}
\end{corollary}
\section{Cuspidal automorphic representations of the special orthogonal group with large discrete series representation at the real place}\label{maindef2}
In this section, we shall define a space of the cuspidal automorphic forms on $\SO_5(\A)$ that generate a large discrete series representation at the real place.
Let $L^2_{\mathrm{disc}}(\SO_5)$ be the discrete spectrum of the unitary representation $L^2(\SO_5(F) \backslash \SO_5(\A_F))$ of $\SO_5(\A_F)$.

Put
\begin{align*}
  \mathcal{H}(\SO_5(\widehat{\Z})\backslash\SO_5(\A_f)/\SO_5(\widehat{\Z}))
  ={\bigotimes_{p<\infty}}' \mathcal{H}(\SO_5(\Z_p)\backslash\SO_5(\Q_p)/\SO_5(\Z_p))
\end{align*}
to be the restricted tensor product with respect to $\{1_{\SO_5(\Z_p)}\}_p$, where $1_{\SO_5(\Z_p)}$ is the characteristic function of $\SO_5(\Z_p)$.
The Hecke algebra acts on $L^2_{\mathrm{disc}}(\SO_5)$ by
\begin{align*}
  [f\cdot\varphi](g) = \int_{\Sp_4(\A)} f(\widetilde{x})\varphi(g\widetilde{x})dx,
\end{align*}
for $f \in \mathcal{H}(\SO_5(\widehat{\Z})\backslash\SO_5(\A_f)/\SO_5(\widehat{\Z}))$, $\varphi \in L^2_{\mathrm{disc}}(\SO_5)$, and $g \in \SO_5(\A)$.

Define a maximal compact torus $S$ of $\SO_5(\R)$ by
\begin{align*}
  S
   & =\Set{
    s(\theta_1, \theta_2)
    =\left(\begin{array}{ccccc}
             \cos\theta_1  & \sin\theta_1 &   &               &              \\
             -\sin\theta_1 & \cos\theta_1 &   &               &              \\
                           &              & 1 &               &              \\
                           &              &   & \cos\theta_2  & \sin\theta_2 \\
                           &              &   & -\sin\theta_2 & \cos\theta_2
           \end{array}
    \right)
    \in \SO_5(\R)
    |
    \theta_1, \theta_2 \in \R
  },
\end{align*}
and fix a basis $\{b_1, b_2\}$ of the group $X^*(S)$ of characters of $S$ given by
\begin{align*}
  b_l(s(\theta_1, \theta_2))
   & = e^{i\theta_l}.
\end{align*}
Let $\sigma_{(\lambda_1,\lambda_2)}$ denote the discrete series representation of $\SO_5(\R)$ with the Blattner parameter $(\lambda_1,\lambda_2)=\lambda_1b_1+\lambda_2b_2$.
\begin{definition}
  Let $k\geq3$ and $j\geq0$ be integers.
  We define
  \begin{equation*}
    \mathcal{A}_{\det^{-j-1}\Sym_{2j+2k-2}}(\SO_5)^{\mathrm{unr}} = \Hom_{\SO_5(\R)}(\sigma_{(j+k-1,k-2)}, L^2_{\mathrm{disc}}(\SO_5)^{\mathrm{unr}}),
  \end{equation*}
  where $L^2_{\mathrm{disc}}(\SO_5)^{\mathrm{unr}}$ denotes the subspace of $L^2_{\mathrm{disc}}(\SO_5)$ consisting of right $\SO_5(\widehat{\Z})$-invariant functions.
\end{definition}

Consider the action of the Hecke algebra on the vector space $\mathcal{A}_{\det^{-j-1}\Sym_{2j+2k-2}}(\SO_5)^{\mathrm{unr}}$ given by
\begin{align*}
  [f\cdot\Psi](v) = f\cdot[\Psi(v)],
\end{align*}
for $f \in \mathcal{H}(\SO_5(\widehat{\Z})\backslash\SO_5(\A_f)/\SO_5(\widehat{\Z}))$, $\Psi\in\mathcal{A}_{\det^{-j-1}\Sym_{2j+2k-2}}(\SO_5)^{\mathrm{unr}}$, and $v\in\sigma_{(j+k-1,k-2)}$.
Then, the $\C$-vector space $\mathcal{A}_{\det^{-j-1}\Sym_{2j+2k-2}}(\SO_5)^{\mathrm{unr}}$ can be decomposed into a direct sum of 1-dimensional eigenspaces.
For a vector $\Psi\in\mathcal{A}_{\det^{-j-1}\Sym_{2j+2k-2}}(\SO_5)^{\mathrm{unr}}$, let $\pi_\Psi$ denote the automorphic representation of $\SO_5(\A)$ generated by the image of $\Psi$.
If $\Psi$ is a Hecke eigenform, then $\pi_\Psi$ is irreducible.
Let $\phi_{\Psi,p}$ be the $L$-parameter corresponding to $\pi_{\Psi,p}$ via LLC for $\SO_5(\Q_p)$ (\cite[Theorem 1.5.1 (b)]{art}).
We now define the $L$-function $L(s,\Psi)$ of a Hecke eigenform $\Psi$ as an Euler product
\begin{equation*}
  L(s,\Psi)
  =\prod_p L(s,\phi_{\Psi,p}).
\end{equation*}
\section{Main theorem on cusp forms on the special orthogonal group for large discrete series representations}\label{mainthm2}
Now we can state our main theorem on cuspidal automorphic forms on $\SO_5(\A)$ generating a large discrete series representation at the real place.
\begin{theorem}\label{main2}
  For any integers $k\geq3$ and $j\geq0$, there exists a linear isomorphism
  \begin{equation*}
    \rho' : \mathcal{A}_{\det^{-j-1}\Sym_{2j+2k-2}}(\SO_5)^{\mathrm{unr}} \overset{\simeq}{\lra} \begin{cases*}
      \mathcal{A}_{\det^{-k+\frac{5}{2}}\Sym_{2k+j-3}}(\Mp_4)^{E_\psi},   & if $j$ is odd and $k$ is even,     \\
      \mathcal{A}_{\det^{-k-j+\frac{1}{2}}\Sym_{2k+j-3}}(\Mp_4)^{E_\psi}, & if $j$ is odd and $k$ is also odd, \\
      S_{\det^{k-\frac{1}{2}}\Sym_j}^+(\Gamma_0(4)),                      & if $j$ is even,
    \end{cases*}
  \end{equation*}
  such that if $\Psi \in \mathcal{A}_{\det^{-j-1}\Sym_{2j+2k-2}}(\SO_5)^{\mathrm{unr}}$ is a Hecke eigenform, then so is $\rho'(\Psi)$, and they satisfy
  \begin{align*}
    L(s,\rho'(\Psi))
    =\begin{cases*}
       L(s,\Psi),                  & if $j$ is odd,  \\
       L(s-j-k+\frac{3}{2}, \Psi), & if $j$ is even.
     \end{cases*}
  \end{align*}
\end{theorem}

We have the following corollaries.
\begin{corollary}
  For any integers $k\geq3$ and $j\geq0$, there exists an injective linear map
  \begin{equation*}
    \mathscr{L}' : S_{2k-4}(\SL_2(\Z)) \otimes S_{2k+2j-2}(\SL_2(\Z)) \lra
    \mathcal{A}_{\det^{-j-1}\Sym_{2j+2k-2}}(\SO_5)^{\mathrm{unr}}
  \end{equation*}
  such that if $f \in S_{2k-4}(\SL_2(\Z))$ and $g \in S_{2k+2j-2}(\SL_2(\Z))$ are Hecke eigenforms, then so is $\mathscr{L}'(f\otimes g)$, and they satisfy
  \begin{align*}
    L(s, \mathscr{L}'(f\otimes g)) = L(s+k-\tfrac{5}{2}, f) L(s+j+k-\tfrac{3}{2},g).
  \end{align*}
\end{corollary}
\begin{proof}
  If $j$ is odd (resp. even), then $\mathscr{L}'$ is given by the composition of $(\rho')\inv$ and $\mathscr{L}$ (Theorem \ref{lifting}) (resp. $\mathcal{L}$ (\cite[Theorem 2.2]{ishi})).
\end{proof}
Let $(\myim\mathscr{L}')^\perp$ denote the orthogonal complement of the image of the injective map $\mathscr{L}'$.
If $k=3$ and $j$ is odd, let $\mathcal{A}^0_{\det^{-\frac{1}{2}}\Sym_{j+3}}(\Mp_4)^{E_\psi}$ denote the image of the orthogonal complement of the Saito-Kurokawa lifting in $S_{\det^{j+3}}(\Sp_4(\Z))$ under the map $\rho^\itN$.
Otherwise, we shall write $\mathcal{A}^0_{\det^{-\frac{1}{2}}\Sym_{j+3}}(\Mp_4)^{E_\psi}$ for $\mathcal{A}_{\det^{-\frac{1}{2}}\Sym_{j+3}}(\Mp_4)^{E_\psi}$ itself.
\begin{corollary}
  For any integers $k\geq3$ and $j\geq0$, there exists a linear isomorphism
  \begin{equation*}
    \breve{\rho} : (\myim\mathscr{L})^\perp \overset{\simeq}{\lra}  \begin{cases*}
      \mathcal{A}_{\det^{-k-j+\frac{1}{2}}\Sym_{2k+j-3}}(\Mp_4)^{E_\psi},            & if $j$ is odd and $k$ is even,     \\
      \mathcal{A}^0_{\det^{-k+\frac{5}{2}}\Sym_{2k+j-3}}(\Mp_4)^{E_\psi},            & if $j$ is odd and $k$ is also odd, \\
      S_{\det^{k-\frac{1}{2}}\Sym_j}^+(\Gamma_0(4),\left( \frac{-1}{\cdot} \right)), & if $j$ is even,
    \end{cases*}
  \end{equation*}
  such that if $\Psi \in (\myim\mathscr{L})^\perp$ is a Hecke eigenform, then so is $\breve{\rho}(\Psi)$, and they satisfy
  \begin{equation*}
    L(s,\breve{\rho}(\Psi))
    =\begin{cases*}
      L(s,\Psi),                  & if $j$ is odd,  \\
      L(s-j-k+\frac{3}{2}, \Psi), & if $j$ is even.
    \end{cases*}
  \end{equation*}
\end{corollary}
\begin{proof}
  If $j$ is odd (resp. even), then $\breve{\rho}$ is given by the composition of $\rho^\itN$ (Theorem \ref{main}) (resp. $(\rho^\itN)\inv$ (\cite[Theorem 2.1]{ishi})), $\rho^\itH$ (Theorem \ref{compl}) (resp. $(\rho^\itH)\inv$ (\cite[Theorem 2.3]{ishi})), and $\rho'$.
\end{proof}


\section{Multiplicity formulae for $\SO_5$ and $\Mp_4$}\label{MF}
In this section, we review the multiplicity formulae for the split $\SO_5$ and $\Mp_4$ briefly.
See \cite{art} (for the split $\SO_5$) and \cite{gi18,gi20} (for $\Mp_4$) for mor detail.
As in \cite{ishi}, they play important roles in this paper.
\subsection{Elliptic $A$-parameters}
Let $F$ be a number field and $\psi$ a nontrivial additive character of $F\backslash \A_F$.
The notion of elliptic $A$-parameters for $\SO_5$ and for $\Mp_4$ are the same.
Recall from \cite{art, gi18, gi20} that an elliptic $A$-parameter for $\SO_5$ or $\Mp_4$ is a formal unordered finite direct sum
\begin{align*}
  \phi
  =\bigoplus_i \phi_i \boxtimes S_{d_i},
\end{align*}
where
\begin{itemize}
  \item each $\phi_i$ is an irreducible self-dual cuspidal automorphic representation of $\GL_{n_i}(\A_F)$;
  \item $S_d$ denotes the unique irreducible $d$-dimensional representation of $\SL_2(\C)$;
  \item if $d_i$ is odd, then $\phi_i$ is symplectic;
  \item if $d_i$ is even, then $\phi_i$ is orthogonal;
  \item if $(\phi_i, d_i)=(\phi_j, d_j)$, then $i=j$;
  \item $\sum_i n_i d_i = 4$.
\end{itemize}
Moreover, if $d_i=1$ for all $i$, then we say that $\phi$ is tempered or generic.

The global component group $S_\phi$ of an elliptic $A$-parameter $\phi=\bigoplus_i \phi_i \boxtimes S_{d_i}$ is defined as a free $\Z/2\Z$-module
\begin{align*}
  S_\phi
  =\bigoplus_i (\Z/2\Z) a_i
\end{align*}
with a formal basis $\{a_i\}$ such that $a_i$ corresponds to $\phi_i\boxtimes S_{d_i}$.
Recall also that Arthur \cite{art} associated a character $\epsilon_\phi$ of $S_\phi$ with $\phi$, and Gan-Ichino \cite{gi20} did another character $\wtil{\epsilon}_\phi$.
\subsection{Local $A$-packets}
For an elliptic $A$-parameter $\phi=\bigoplus_i \phi_i \boxtimes S_{d_i}$ for $\SO_5$ or $\Mp_4$ and a place $v$ of $F$, the localization $\phi_v : L_{F_v} \times \SL_2(\C) \to \Sp_4(\C)$ is defined by
\begin{align*}
  \phi_v
  = \bigoplus_i \phi_{i, v} \boxtimes S_{d_i}.
\end{align*}
Here, the irreducible representation $\phi_{i,v}$ of $\GL_{n_i}(F_v)$ is identified with an $L$-parameter $L_{F_v} \to \GL_{n_i}(\C)$ via the local Langlands correspondence \cite{lan, ht, hen}.
For such $\phi_v$, the associated $L$-parameter $\varphi_{\phi_v} : L_{F_v} \to \Sp_4(\C)$ is defined by
\begin{align*}
  \varphi_{\phi_v} (w)
  =\phi_v(w, \mmatrix{|w|_v^{\frac{1}{2}}}{}{}{|w|_v^{-\frac{1}{2}}}).
\end{align*}
We let $S_{\phi_v}$ (resp. $S_{\varphi_{\phi_v}}$) denote the component group of the centralizer of the image of $\phi_v$ (resp. $\varphi_{\phi_v}$) in $\Sp_4(\C)$.
Note that there exists a natural map $S_\phi \to S_{\phi_v}$.

Let us write $\Pi_{\phi_v}(\SO_5)$ and $\Pi_{\phi_v, \psi_v}(\Mp_4)$ for the local $A$-packets associated to $\phi_v$ (relative to $\psi_v$ in the case of $\Mp_4$).
Recall that $\Pi_{\phi_v}(\SO_5)$, given by Arthur \cite{art}, is a finite set
\begin{align*}
  \Pi_{\phi_v}(\SO_5)
  =\Set{\sigma_{\eta_v} | \eta_v \in \what{S}_{\phi_v}}
\end{align*}
of semisimple representations of $\SO_5(F_v)$ of finite length indexed by characters of $S_{\phi_v}$, and $\Pi_{\phi_v, \psi_v}(\Mp_4)$, given by Gan-Ichino \cite{gi18, gi20}, is a finite set
\begin{align*}
  \Pi_{\phi_v, \psi_v}(\Mp_4)
  =\Set{\pi_{\eta_v}=\pi_{\eta_v, \psi_v} | \eta_v \in \what{S}_{\phi_v}}
\end{align*}
of semisimple genuine representations of $\Mp_4(F_v)$ of finite length indexed by characters of $S_{\phi_v}$.
Note that the latter packet $\Pi_{\phi_v, \psi_v}(\Mp_4)$ depends on the additive character $\psi_v$ of $F_v$.
Recall also that
\begin{itemize}
  \item if $v$ is a finite place and $\sigma_{\eta_v}$ (resp. $\pi_{\eta_v', \psi_v}$) has an unramified constituent, then $\phi_{i,v}$ are unramified and $\eta_v$ (resp. $\eta_v'$) is a trivial character;
  \item if $\phi$ is tempered, then the local $A$-packets $\Pi_{\phi_v}(\SO_5)$ and $\Pi_{\phi_v, \psi_v}(\Mp_4)$ coincide with $L$-packets associated to the $L$-parameter $\varphi_{\phi_v}=\phi_v$ given by \cite{art, ab95, ab98, gs}.
\end{itemize}

For later convenience, let us recall here from \cite[\S C.2.1]{gi20} a classification of discrete series representations of $\Mp_4(\R)$.
For $(a,b)\in (\frac{1}{2}+\Z)^2$ with $a> b>0$, consider an $L$-parameter $\varphi=\mathcal{D}_a\oplus\mathcal{D}_b$.
The component group $S_\varphi$ is isomorphic to $(\Z/2\Z)x_1\oplus(\Z/2\Z)x_2$, where $\{x_1,x_2\}$ is a formal basis and $x_1$ (resp. $x_2$) corresponds to $\mathcal{D}_a$ (resp. $\mathcal{D}_b$).
Then the $L$-packet is
\begin{equation*}
  \Pi_{\varphi,\psi}(\Mp_4(\R)) = \Set{\pi_\varphi^{\epsilon_1,\epsilon_2}=\pi_{\Lambda^{\epsilon_1,\epsilon_2}} | \epsilon_1,\epsilon_2\in\{\pm\}},
\end{equation*}
where
\begin{align*}
  \Lambda^{+,+} & =(a+1,-b),    \\
  \Lambda^{+,-} & =(a+1,b+2),   \\
  \Lambda^{-,+} & =(-b-2,-a-1), \\
  \Lambda^{-,-} & =(b,-a-1),
\end{align*}
and $\pi^{\epsilon_1,\epsilon_2}$ corresponds to the character sending $x_i$ to $\epsilon_i$.
These exhaust all irreducible genuine discrete series representations of $\Mp_4(\R)$.
\subsection{Multiplicity formulae}
As in Sections \ref{maindef} and \ref{maindef2}, we shall write $L^2_{\mathrm{disc}}(\Mp_4)$ and $L^2_{\mathrm{disc}}(\SO_5)$ for the discrete spectrum of the (genuine) unitary representations $L^2(\Mp_4)$ and $L^2(\SO_5(F) \backslash \SO_5(\A_F))$ of $\Mp_4(\A_F)$ and $\SO_5(\A_F)$, respectively.
The multiplicity formulae describe the decomposition of $L^2_{\mathrm{disc}}(\Mp_4)$ and $L^2_{\mathrm{disc}}(\SO_5)$ into near equivalence classes of irreducible representations, and the multiplicity of any irreducible representation in each near equivalence classes.
Here we say that two irreducible representations $\pi=\otimes_v\pi_v$ and $\pi'=\otimes_v\pi_v'$ of $\Mp_4(\A_F)$ or $\SO_5(\A_F)$ are nearly equivalent if $\pi_v$ and $\pi_v'$ are equivalent for almost all places $v$ of $F$.

The adelic component group $S_{\phi, \A}$ of an elliptic $A$-parameter $\phi$ for $\SO_5$ or $\Mp_4$ is defined by the infinite direct product
\begin{align*}
  S_{\phi, \A}
  =\prod_v S_{\phi_v},
\end{align*}
and we shall write $\Delta$ for the diagonal map $S_\phi \to S_{\phi, \A}$.
Let $\what{S}_{\phi, \A}$ be the group of continuous characters of $S_{\phi, \A}$.
Note that $\what{S}_{\phi, \A}= \bigoplus_v \what{S}_{\phi_v}$.
For any $\eta=\bigotimes_v \eta_v \in \what{S}_{\phi, \A}$, let $\sigma_\eta$ and $\pi_\eta=\pi_{\eta, \psi}$ denote semisimple representations
\begin{align*}
  \sigma_\eta      & = \bigotimes_v \sigma_{\eta_v},      &
  \pi_{\eta, \psi} & = \bigotimes_v \pi_{\eta_v, \psi_v},
\end{align*}
of $\SO_5(\A_F)$ and $\Mp_4(\A_F)$, respectively.
Then the multiplicity formulae for $\SO_5$ and $\Mp_4$ can be stated as follows.
\begin{theorem}[{\cite[Theorem 1.5.2]{art}}]\label{amf}
  For every elliptic $A$-parameter $\phi$ for $\SO_5$, put
  \begin{align*}
    L^2_\phi(\SO_5) & =\bigoplus_{\eta \in \what{S}_{\phi, \A}} n_\eta \sigma_\eta, \\
    n_\eta          & =
    \begin{cases}
      1, & \text{if $\eta \circ \Delta = \epsilon_\phi$}, \\
      0, & \text{otherwise}.
    \end{cases}
  \end{align*}
  Then each $L^2_\phi(\SO_5)$ is a full near equivalence class of irreducible representations in the discrete spectrum $L^2_{\mathrm{disc}}(\SO_5)$, and $L^2_{\mathrm{disc}}(\SO_5)$ can be decomposed into a direct sum
  \begin{align*}
    L^2_{\mathrm{disc}}(\SO_5)
    =\bigoplus_\phi L^2_\phi(\SO_5).
  \end{align*}
\end{theorem}

\begin{theorem}[{\cite[Theorem 2.1]{gi20}}]\label{gimf}
  For every elliptic $A$-parameter $\phi$ for $\Mp_4$, put
  \begin{align*}
    L^2_{\phi,\psi}(\Mp_4) & =\bigoplus_{\eta \in \what{S}_{\phi, \A}} m_\eta \pi_{\eta, \psi}, \\
    m_\eta                 & =
    \begin{cases}
      1, & \text{if $\eta \circ \Delta = \wtil{\epsilon}_\phi$}, \\
      0, & \text{otherwise}.
    \end{cases}
  \end{align*}
  Then each $L^2_{\phi,\psi}(\Mp_4)$ is a full near equivalence class of irreducible representations in the discrete spectrum $L^2_{\mathrm{disc}}(\Mp_4)$, and $L^2_{\mathrm{disc}}(\Mp_4)$ can be decomposed into a direct sum
  \begin{align*}
    L^2_{\mathrm{disc}}(\Mp_4)
    =\bigoplus_\phi L^2_{\phi,\psi}(\Mp_4).
  \end{align*}
\end{theorem}

\section{Proofs of the theorems}\label{pf}
In this section, we will prove the main theorems.
The idea is similar to those of \cite{ishi}.
\subsection{The $A$-parameters associated to Siegel cusp forms of integral weight of degree 2}\label{adlifato}
We shall first consider the adelic lifts of Siegel cusp forms of integral weight.
Let $k\geq3$ be an integer and $j\geq0$ an odd integer.
Let $F \in S_{\det^{j+3}\Sym_{2k-6}}(\Sp_4(\Z))$ be a Hecke eigenform.
As in \cite[\S 4.5, \S 6.3.4]{cl}, we construct an irreducible cuspidal automorphic representation $\pi_F$ of $\PGSp_4(\A)$, which is unramified everywhere.
By an accidental isomorphism $\PGSp_4\cong \SO_5$ (\cite[Lemma 6.1]{ishi}), we consider $\pi_F$ as a representation of $\SO_5(\A)$.
Let $\phi_F$ be the $A$-parameter of $\pi_F$.
Note that by Theorem \ref{amf}, the elliptic $A$-parameter of $\pi_F$ exists uniquely.

Let us determine the form of the $A$-parameter $\phi_F$.
We have the following lemma.
\begin{lemma}\label{api}
  If $k>3$ or $k=3$ and $F$ is not a Saito-Kurokawa lift, then there is a unique irreducible everywhere unramified cuspidal symplectic automorphic representation $\tau_F$ of $\GL_4(\A)$ such that $\phi_F=\tau_F\boxtimes S_1$ and
  \begin{align*}
    \begin{aligned}
      L(s,F,\mathrm{spin})
       & =L(s-j-k+\tfrac{3}{2}, \tau_F),                        \\
      \tau_{F,\infty}
       & =\calD_{k+j-\frac{3}{2}} \oplus \calD_{k-\frac{5}{2}}.
    \end{aligned}
  \end{align*}
  In the case that $k=3$ and $F$ is a Saito-Kurokawa lift, there is a unique irreducible everywhere unramified cuspidal symplectic automorphic representation $\sigma_F$ of $\GL_2(\A)$ such that $\phi_F=1\boxtimes S_2 \oplus \sigma_F\boxtimes S_1$ and
  \begin{align*}
    \begin{aligned}
      L(s,F,\mathrm{spin})
       & =\zeta(s-j-2) \zeta(s-j-1) L(s-j-\tfrac{3}{2}, \sigma_F), \\
      \sigma_{F,\infty}
       & =\mathcal{D}_{j+\frac{3}{2}},
    \end{aligned}
  \end{align*}
  where $\zeta(s)$ denotes the Riemann zeta function.
\end{lemma}
\begin{proof}
  The idea of the proof is the same as that of \cite[Lemma 6.2]{ishi}

  Recall that in Section \ref{maindef2}, we have defined a maximal compact torus $S$ of $\SO_5(\R)$ and fixed a basis $\{b_1, b_2\}$ of $X^*(S)$.
  Then the real component $\pi_{F, \infty}$ of $\pi_F$ is a discrete series representation with the Blattner parameter $(k-3, k+j)$, and its $L$-parameter is $\calD_{k+j-\frac{3}{2}} \oplus \calD_{k-\frac{5}{2}}$.
  By \cite{mr1, mr2}, if $k>3$ the only local $A$-parameter such that its local $A$-packet contains $\pi_{F, \infty}$ is
  \begin{align*}
    \calD_{k+j-\frac{3}{2}} \boxtimes S_1 \oplus \calD_{k-\frac{5}{2}}\boxtimes S_1.
  \end{align*}
  On the other hand, if $k=3$, the local $A$-parameters of which the local $A$-packets contain $\pi_{F, \infty}$ are
  \begin{align}\label{rp2}
    \calD_{j+\frac{3}{2}} \boxtimes S_1 & \oplus \calD_{\frac{1}{2}}\boxtimes S_1, &
    \calD_{j+\frac{3}{2}} \boxtimes S_1 & \oplus \chi_a \boxtimes S_2,
  \end{align}
  where $\chi_a$ is a quadratic character of $W_\R$.
  In the packet $\Pi_{\calD_{j+\frac{3}{2}} \boxtimes S_1 \oplus \calD_{\frac{1}{2}}\boxtimes S_1}(\SO_5)$, the representation $\pi_{F, \infty}$ corresponds to the character
  \begin{align}\label{tukao}
    S_{\calD_{j+\frac{3}{2}} \boxtimes S_1 \oplus \calD_{\frac{1}{2}}\boxtimes S_1} \cong (\Z/2\Z)^2
    \ni (c,d) \mapsto (-1)^{c+d}.
  \end{align}
  In the packet $\Pi_{\calD_{j+\frac{3}{2}} \boxtimes S_1 \oplus \chi_a \boxtimes S_2}(\SO_5)$, the representation $\pi_{F, \infty}$ also corresponds to the character
  \begin{align}\label{ch1}
    S_{\calD_{j+\frac{3}{2}} \boxtimes S_1 \oplus \chi_a \boxtimes S_2} \cong (\Z/2\Z)^2
    \ni (c,d) \mapsto (-1)^{c+d}.
  \end{align}

  The $A$-parameter $\phi_F$ for $\SO_5$ is by definition one of the following forms:
  \begin{enumerate}[(i)]
    \item $\phi_F=\chi \boxtimes S_4$, where $\chi$ is a quadratic character of $\Q^\times \backslash \A^\times$;
    \item $\phi_F=\chi \boxtimes S_2 \oplus \chi' \boxtimes S_2$, where $\chi$ and $\chi'$ are distinct quadratic characters of $\Q^\times \backslash \A^\times$;
    \item $\phi_F=\sigma \boxtimes S_2$, where $\sigma$ is an irreducible cuspidal orthogonal automorphic representation of $\GL_2(\A)$;
    \item $\phi_F=\chi \boxtimes S_2 \oplus \sigma \boxtimes S_1$, where $\chi$ is a quadratic character of $\Q^\times \backslash \A^\times$, and $\sigma$ is an irreducible cuspidal symplectic automorphic representation of $\GL_2(\A)$;
    \item $\phi_F=\sigma \boxtimes S_1 \oplus \sigma' \boxtimes S_1$, where $\sigma$ and $\sigma'$ are distinct irreducible cuspidal symplectic automorphic representations of $\GL_2(\A)$;
    \item $\phi_F=\tau \boxtimes S_1$, where $\tau$ is an irreducible cuspidal symplectic automorphic representation of $\GL_4(\A)$.
  \end{enumerate}

  Since $\pi_F$ is unramified everywhere, $\chi$, $\chi'$, $\sigma$, $\sigma'$, and $\tau$ are also unramified everywhere.
  Moreover, $\chi$ must be trivial since $\A^\times =\Q^\times \R^\times_{>0} \hat{\Z}^\times$.
  When $k>3$, by \cite[Proposition 9.1.4]{cl}, we have $\phi_F=\tau \boxtimes S_1$ as in the case (vi).
  Let us consider the case that $k$ is 3.
  Since the localization $\phi_{F, \infty}$ at the real place must be one of the forms \eqref{rp2}, the cases (i), (ii), and (iii) cannot occur.

  Suppose that $\phi_F$ is of the form (v).
  The character of the local component group corresponding to $\pi_{F,v}$ is \eqref{tukao} at the real place, and is trivial at every finite places.
  However, since $\phi_F$ is tempered, Arthur's character $\epsilon_{\phi_F}$ is trivial.
  These contradict the multiplicity formula (Theorem \ref{amf}).

  Now consider the case (iv).
  Suppose that $\phi_F$ is of the form $1 \boxtimes S_2 \oplus \sigma \boxtimes S_1$, where $\sigma$ is an irreducible cuspidal symplectic automorphic representation of $\GL_2(\A)$.
  Then $\sigma$ must be $\mathcal{D}_{j+\frac{3}{2}}$, and there exists a Hecke eigenform $f\in S_{2j+4}(\SL_2(\Z))$ such that $\tau_f=\sigma$.
  This leads to an equation
  \begin{align*}
    L(s,F,\operatorname{spin})= & L(s-j-\tfrac{3}{2},\phi_F)                                                              \\
                                & =\zeta(s-j-2) \zeta(s-j-1) L(s-j-\tfrac{3}{2},\tau_f)=\zeta(s-j-2) \zeta(s-j-1) L(s,f),
  \end{align*}
  which means that $F$ is the Saito-Kurokawa lift of $f$.

  Conversely, if $F$ is the Saito-Kurokawa lift of $f\in S_{2j+4}(\SL_2(\Z))$, then we have
  \begin{align*}
    L(s,\phi_F)= & L(s+j+\tfrac{3}{2},F,\operatorname{spin})                                                                                   \\
                 & =\zeta(s-\tfrac{1}{2}) \zeta(s+\tfrac{1}{2}) L(s+j+\tfrac{3}{2},f)=\zeta(s-\tfrac{1}{2}) \zeta(s+\tfrac{1}{2}) L(s,\tau_f),
  \end{align*}
  where $\tau_f$ is the cuspidal automorphic representation of $\GL_2(\A)$ corresponding to $f$.
  Therefore, we obtain that $\phi_F=1\boxtimes S_2 \oplus \tau_f\boxtimes S_1$.
  The localization of $\phi_F$ at the real place must be $\phi_{F, \infty}=1\boxtimes S_2 \oplus \calD_{j+\frac{3}{2}} \boxtimes S_1$.\\

  In conclusion, if $F$ is not a Saito-Kurokawa lift, we have $\phi_F = \tau_F \boxtimes S_1$, where $\tau_F$ is an irreducible cuspidal symplectic automorphic representation of $\GL_4(\A)$ such that
  \begin{align*}
    \tau_{F,\infty}
     & =\calD_{k+j-\frac{3}{2}} \oplus \calD_{k-\frac{5}{2}}, \\
    L(s,F,\mathrm{spin})
     & =L(s-j-k+\tfrac{3}{2}, \tau_F).
  \end{align*}

  The uniqueness follows from that of the $A$-parameter.
\end{proof}

Moreover, we can show the existence and uniqueness of a Hecke eigenform $F$ for a given $A$-parameter $\phi$ of the form above:
\begin{lemma}\label{iap}
  \begin{enumerate}
    \item Let $\tau$ be an irreducible cuspidal symplectic automorphic representation of $\GL_4(\A)$ that is unramified everywhere and satisfies $\tau_\infty=\calD_{k+j-\frac{3}{2}} \oplus \calD_{k-\frac{5}{2}}$.
          Then there exists a Hecke eigenform $F \in S_{\det^{j+3}\Sym_{2k-6}}(\Sp_4(\Z))$ with $\phi_F=\tau\boxtimes S_1$, and it is unique up to a scalar multiple.
    \item Let $\sigma$ be an irreducible cuspidal symplectic automorphic representation of $\GL_2(\A)$ that is unramified everywhere and satisfies $\sigma_\infty=\calD_{j+\frac{3}{2}}$.
          Then there exists a Hecke eigenform $F \in S_{\det^{j+3}}(\Sp_4(\Z))$ with $\phi_F=1\boxtimes S_2 \oplus \sigma\boxtimes S_1$, and it is unique up to a scalar multiple.
  \end{enumerate}
\end{lemma}
\begin{proof}
  The proof of the first assertion is the same as \cite[Proposition 9.1.4 (i) and (iii)]{cl}.
  For the second assertion, consider a Hecke eigenform $f\in S_{2j+4}(\SL_2(\Z))$ corresponding to $\sigma$.
  Here, note that $\sigma$ is unramified everywhere and isomorphic to $\calD_{j+\frac{3}{2}}$ at the real place.
  By the multiplicity one theorem for $\GL_2$, $f$ is unique up to a scalar multiple.
  Let $F\in S_{\det^{j+3}}(\Sp_4(\Z))$ be the Saito-Kurokawa lift of $f$.
  Then the assertion holds for $F$.
\end{proof}

\subsection{The $A$-parameters associated to automorphic representations of $\Mp_4$ generating large discrete series representations at the real place}
Next, we shall consider the automorphic representations of $\Mp_4$.
Let $k\geq3$ be an integer and $j\geq0$ an odd integer as in the previous subsection.
Let $\Phi$ be an element in $\mathcal{A}_{\det^{-k+\frac{5}{2}}\Sym_{2k+j-3}}(\Mp_4)^{E_\psi}$ or $\mathcal{A}_{\det^{-k-j+\frac{1}{2}}\Sym_{2k+j-3}}(\Mp_4)^{E_\psi}$, and suppose that it is a Hecke eigenform.
We shall write $\pi_\Phi$ for the irreducible automorphic representation of $\Mp_4(\A)$ generated by the image of $\Phi$.
Let $\phi_\Phi$ be the $A$-parameter of $\pi_\Phi$ relative to the fixed additive character $\psi$.
One can determine the form of the $A$-parameter $\phi_\Phi$ as follows.
Put $\varphi=\calD_{k+j-\frac{3}{2}} \oplus \calD_{k-\frac{5}{2}}$.
\begin{lemma}\label{ds}
  We have
  \begin{align*}
    \pi_{\Phi,\infty}=\pi_\varphi^{\epsilon,\epsilon},
  \end{align*}
  where
  \begin{align*}
    \epsilon=
    \begin{cases*}
      +, & if $\Phi \in \mathcal{A}_{\det^{-k+\frac{5}{2}}\Sym_{2k+j-3}}(\Mp_4)^{E_\psi}$,   \\
      -, & if $\Phi \in \mathcal{A}_{\det^{-k-j+\frac{1}{2}}\Sym_{2k+j-3}}(\Mp_4)^{E_\psi}$.
    \end{cases*}
  \end{align*}
\end{lemma}
\begin{proof}
  If $\Phi \in \mathcal{A}_{\det^{-k+\frac{5}{2}}\Sym_{2k+j-3}}(\Mp_4)^{E_\psi}$, we have $\pi_{\Phi,\infty}=\pi_{(k+j-\frac{1}{2}, -k+\frac{5}{2})}=\pi_\varphi^{+,+}$.

  If $\Phi \in \mathcal{A}_{\det^{-k-j+\frac{1}{2}}\Sym_{2k+j-3}}(\Mp_4)^{E_\psi}$, we have $\pi_{\Phi,\infty}=\pi_{(k-\frac{5}{2}, -k-j+\frac{1}{2})}=\pi_\varphi^{-,-}$.
\end{proof}

\begin{lemma}\label{ap2}
  Suppose that $\Phi \in \mathcal{A}_{\det^{-k+\frac{5}{2}}\Sym_{2k+j-3}}(\Mp_4)^{E_\psi}$.
  If $k>3$ and is odd, then there is a unique irreducible everywhere unramified cuspidal symplectic automorphic representation $\tau_\Phi$ of $\GL_4(\A)$ such that $\phi_\Phi=\tau_\Phi\boxtimes S_1$ and
  \begin{align*}
    L(s,\Phi)
     & =L(s, \tau_\Phi),                                      \\
    \tau_{\Phi,\infty}
     & =\calD_{k+j-\frac{3}{2}} \oplus \calD_{k-\frac{5}{2}}.
  \end{align*}

  If $k$ is even, then there exists one and only one of such a $\tau_\Phi$ or a pair $(\sigma_\Phi, \sigma_\Phi')$ of irreducible everywhere unramified cuspidal symplectic automorphic representations $\sigma_\Phi$ and $\sigma_\Phi'$ of $\GL_2(\A)$ such that $\phi_\Phi=\sigma_\Phi \boxtimes S_1 \oplus \sigma_\Phi'\boxtimes S_1$ and
  \begin{gather*}
    L(s,\Phi)
    =L(s, \sigma_\Phi)L(s, \sigma_\Phi'),\\
    \sigma_{\Phi,\infty}
    =\calD_{k+j-\frac{3}{2}}, \qquad
    \sigma_{\Phi,\infty}'
    =\calD_{k-\frac{5}{2}}.
  \end{gather*}

  If $k=3$, then there exists one and only one of such a $\tau_\Phi$ or irreducible everywhere unramified cuspidal symplectic automorphic representation $\sigma_\Phi$ of $\GL_2(\A)$ such that $\phi_\Phi=1\boxtimes S_2 \oplus \sigma_\Phi\boxtimes S_1$ and
  \begin{align*}
    \begin{aligned}
      L(s,\Phi)
       & =\zeta(s+\tfrac{1}{2}) \zeta(s-\tfrac{1}{2}) L(s, \sigma_\Phi), \\
      \sigma_{\Phi,\infty}
       & =\mathcal{D}_{j+\frac{3}{2}}.
    \end{aligned}
  \end{align*}
\end{lemma}
\begin{proof}
  The $A$-parameter $\phi_\Phi$ (of $\pi_\Phi$ relative to $\psi$) is one of the following forms:
  \begin{enumerate}[(i)]
    \item $\phi_\Phi=\chi \boxtimes S_4$, where $\chi$ is a quadratic character of $\Q^\times \backslash \A^\times$;
    \item $\phi_\Phi=\chi \boxtimes S_2 \oplus \chi' \boxtimes S_2$, where $\chi$ and $\chi'$ are distinct quadratic characters of $\Q^\times \backslash \A^\times$;
    \item $\phi_\Phi=\sigma \boxtimes S_2$, where $\sigma$ is an irreducible cuspidal orthogonal automorphic representation of $\GL_2(\A)$;
    \item $\phi_\Phi=\chi \boxtimes S_2 \oplus \sigma \boxtimes S_1$, where $\chi$ is a quadratic character of $\Q^\times \backslash \A^\times$, and $\sigma$ is an irreducible cuspidal symplectic automorphic representation of $\GL_2(\A)$;
    \item $\phi_\Phi=\sigma \boxtimes S_1 \oplus \sigma' \boxtimes S_1$, where $\sigma$ and $\sigma'$ are distinct irreducible cuspidal symplectic automorphic representations of $\GL_2(\A)$;
    \item $\phi_\Phi=\tau \boxtimes S_1$, where $\tau$ is an irreducible cuspidal symplectic automorphic representation of $\GL_4(\A)$.
  \end{enumerate}

  Since $\pi_\Phi$ is unramified everywhere, $\chi$, $\chi'$, $\sigma$, $\sigma'$, and $\tau$ are also unramified everywhere.
  Moreover, $\chi$ must be trivial since $\A^\times =\Q^\times \R^\times_{>0} \hat{\Z}^\times$.
  If the case is (i) or (ii), then by \cite{gi20}, the local $A$-packet $\Pi_{\phi_{\Phi, \infty}, \psi_\infty}(\Mp_4(\R))$ does not contain any discrete series representation.
  This contradicts Lemma \ref{ds}, and thus the cases (i) and (ii) cannot occur.

  Let us consider the case (iii).
  The local component $\sigma_\infty$ of $\sigma$ at the real place must be irreducible by the same reason why the case (ii) cannot occur.
  Then by the table in \cite[\S C.2.2]{gi20}, the discrete series representation $\pi_{\Phi,\infty} \in \Pi_{\phi_{\Phi, \infty}, \psi_\infty}(\Mp_4(\R))$ corresponds to the nontrivial character of $S_{\phi_{\Phi, \infty}} \cong \Z/2\Z$.
  On the other hand, local components $\pi_{\Phi,p}$ at the finite places are unramified.
  These contradict the fact that Gan-Ichino's character $\wtil{\epsilon}_{\phi_\Phi}$ \cite[\S 2.1]{gi20} is trivial, so the case (iii) is impossible.

  Assume that $\phi_\Phi$ is of the form (iv).
  In this case we have $S_{\phi_\Phi} = (\Z/2\Z)a_\chi \oplus (\Z/2\Z)a_\sigma$, where $a_\diamondsuit$ corresponds to $\diamondsuit$.
  If the real component $\sigma_\infty$ of $\sigma$ is reducible, the local $A$-packet $\Pi_{\phi_{\Phi, \infty}, \psi_\infty}(\Mp_4(\R))$ has two nonzero elements and neither one is a discrete series representation by \cite[\S 8.1, Lemma C8]{gi20}.
  Thus $\sigma_\infty$ is irreducible.
  Since $\chi$ is the trivial character of $\A^\times$, Gan-Ichino's character $\wtil{\epsilon}_{\phi_\Phi}$ is trivial on $(\Z/2\Z)a_\sigma \subset S_{\phi_F}$.
  Combining this with the table in \cite[\S C.2.2]{gi20}, we obtain that $k=3$, $\sigma_\infty=\calD_{j+\frac{3}{2}}$, and the $A$-packet $\Pi_{\phi_{\Phi, \infty}, \psi_\infty}(\Mp_4(\R))$ contains only one discrete series representation $\pi_\varphi^{+,+}$.
  This is compatible with Lemma \ref{ds}.

  Assume that $\phi_F$ is of the form (v).
  Then by the disjointness of $L$-packets, we have $\phi_{\Phi, \infty}=\varphi$.
  We may assume that
  \begin{align*}
    \sigma_\infty  & =\calD_{k+j-\frac{3}{2}}, &
    \sigma'_\infty & =\calD_{k-\frac{5}{2}}.
  \end{align*}
  Since $\pi_\Phi$ is unramified at every finite place, so are $\sigma$ and $\sigma'$.
  Thus we have
  \begin{align*}
    \epsilon(\tfrac{1}{2}, \sigma)
     & =\epsilon(\tfrac{1}{2}, \calD_{k+j-\frac{3}{2}}, \psi_\infty)
    =i^{2k+2j-3+1}
    =(-1)^k,                                                         \\
    \epsilon(\tfrac{1}{2}, \sigma')
     & =\epsilon(\tfrac{1}{2}, \calD_{k-\frac{5}{2}}, \psi_\infty)
    =i^{2k-5+1}
    =(-1)^k,
  \end{align*}
  since $j$ is odd.
  By Lemma \ref{ds}, we have $\pi_{\Phi,\infty}^{+,+}$.
  Hence $k$ must be even.

  As Lemma \ref{api}, the uniqueness follows from that of the $A$-parameter.
\end{proof}

\begin{lemma}\label{ap3}
  Suppose that $\Phi \in \mathcal{A}_{\det^{-k-j+\frac{1}{2}}\Sym_{2k+j-3}}(\Mp_4)^{E_\psi}$.
  If $k$ is even, then there is a unique irreducible everywhere unramified cuspidal symplectic automorphic representation $\tau_\Phi$ of $\GL_4(\A)$ such that $\phi_\Phi=\tau_\Phi\boxtimes S_1$ and
  \begin{align*}
    L(s,\Phi)
     & =L(s, \tau_\Phi),                                      \\
    \tau_{\Phi,\infty}
     & =\calD_{k+j-\frac{3}{2}} \oplus \calD_{k-\frac{5}{2}}.
  \end{align*}

  If $k$ is odd, then there exists one and only one of such a $\tau_\Phi$ or a pair $(\sigma_\Phi, \sigma_\Phi')$ of irreducible everywhere unramified cuspidal symplectic automorphic representations $\sigma_\Phi$ and $\sigma_\Phi'$ of $\GL_2(\A)$ such that $\phi_\Phi=\sigma_\Phi \boxtimes S_1 \oplus \sigma_\Phi'\boxtimes S_1$ and
  \begin{gather*}
    L(s,\Phi)
    =L(s, \sigma_\Phi)L(s, \sigma_\Phi'),\\
    \sigma_{\Phi,\infty}
    =\calD_{k+j-\frac{3}{2}}, \qquad
    \sigma_{\Phi,\infty}'
    =\calD_{k-\frac{5}{2}}.
  \end{gather*}
\end{lemma}
\begin{proof}
  The proof is similar to that of Lemma \ref{ap2}.
  The only difference is that $\pi_{\Phi,\infty}$ is not $\pi_\varphi^{+,+}$ but $\pi_\varphi^{-,-}$, by Lemma \ref{ds}.
  Because of this difference, the $A$-parameter $\phi_\Phi$ cannot be of the form (iv) and the parity condition on $k$ changes.
\end{proof}

Moreover, we can show the existence and uniqueness of a Hecke eigenform $\Phi$ for a given $A$-parameter $\phi$ of the form above:
\begin{lemma}\label{fap}
  \begin{enumerate}
    \item Let $\tau$ be an irreducible cuspidal symplectic automorphic representation of $\GL_4(\A)$ that is unramified everywhere and satisfies $\tau_\infty=\calD_{k+j-\frac{3}{2}} \oplus \calD_{k-\frac{5}{2}}$.
          Then there exist Hecke eigenforms $\Phi_\mathrm{II} \in \mathcal{A}_{\det^{-k+\frac{5}{2}}\Sym_{2k+j-3}}(\Mp_4)^{E_\psi}$ and $\Phi_\mathrm{III} \in \mathcal{A}_{\det^{-k-j+\frac{1}{2}}\Sym_{2k+j-3}}(\Mp_4)^{E_\psi}$ with $\phi_{\Phi_\mathrm{II}}=\phi_{\Phi_\mathrm{III}}=\tau\boxtimes S_1$, and they are unique up to a scalar multiple.
    \item Let $\sigma$ be an irreducible cuspidal symplectic automorphic representation of $\GL_2(\A)$ that is unramified everywhere and satisfies $\sigma_\infty=\calD_{j+\frac{3}{2}}$.
          Then there exists a Hecke eigenform $\Phi_\mathrm{II} \in \mathcal{A}_{\det^{-k+\frac{5}{2}}\Sym_{2k+j-3}}(\Mp_4)^{E_\psi}$ with $\phi_{\Phi_\mathrm{II}}=1\boxtimes S_2 \oplus \sigma\boxtimes S_1$, and it is unique up to a scalar multiple.
    \item Let $(\sigma,\sigma')$ be a pair of irreducible cuspidal symplectic automorphic representations of $\GL_2(\A)$ that are unramified everywhere and satisfy $\sigma_\infty=\mathcal{D}_{k+j-\frac{3}{2}}$ and $\sigma'_\infty=\mathcal{D}_{k-\frac{5}{2}}$.
          Then if $k$ is even (resp. odd), there exists a Hecke eigenform $\Phi_\mathrm{II} \in \mathcal{A}_{\det^{-k+\frac{5}{2}}\Sym_{2k+j-3}}(\Mp_4)^{E_\psi}$ (resp. $\Phi_\mathrm{III} \in \mathcal{A}_{\det^{-k-j+\frac{1}{2}}\Sym_{2k+j-3}}(\Mp_4)^{E_\psi}$) with $\phi_{\Phi_\mathrm{II}}$ (resp. $\phi_{\Phi_\mathrm{III}}$) $=\sigma\boxtimes S_1 \oplus \sigma'\boxtimes S_1$, and it is unique up to a scalar multiple.
  \end{enumerate}
\end{lemma}
\begin{proof}
  \begin{enumerate}
    \item Put $\phi=\tau\boxtimes S_1$. For any prime $p<\infty$, let us write $\pi_p$ for the unique $\psi_p$-unramified representation in $\Pi_{\phi_p,\psi_p}(\Mp_4(\Q_p))$, which corresponds to the trivial character of $S_{\phi_p}$. Since
          \begin{align*}
            \epsilon(\tfrac{1}{2},\tau)=\epsilon(\tfrac{1}{2},\calD_{k+j-\frac{3}{2}},\psi_\infty) \epsilon(\tfrac{1}{2},\calD_{k-\frac{5}{2}},\psi_\infty)=+1,
          \end{align*}
          Gan-Ichino's character $\widetilde{\epsilon}_\phi$ is trivial. Thus representations $\pi_{\phi_\infty}^{+,+}\otimes\bigotimes_p \pi_p$ and $\pi_{\phi_\infty}^{-,-}\otimes\bigotimes_p \pi_p$ are automorphic. They determine $\Phi_\mathrm{II}$ and $\Phi_\mathrm{III}$ up to scalar respectively, since $\pi_{\phi_\infty}^{+,+}=\pi_{(k+j-\frac{1}{2},-k+\frac{5}{2})}$ and $\pi_{\phi_\infty}^{-,-}=\pi_{(k-\frac{5}{2},-k-j+\frac{1}{2})}$. The uniqueness follows from the uniqueness of local components $\pi_{\phi_\infty}^{\epsilon,\epsilon}$ and $\pi_p$ in the respective local $L$-packets.
    \item Put $\phi=1\boxtimes S_2 \oplus \sigma\boxtimes S_1$ and write $S_\phi=(\Z/2\Z)a_1 \oplus (\Z/2\Z)a_\sigma$, where $a_1$ (resp. $a_\sigma$) corresponds to $1\boxtimes S_2$ (resp. $\sigma\boxtimes S_1$). For any prime $p<\infty$, let us write $\pi_p$ for the unique $\psi_p$-unramified representation in $\Pi_{\phi_p,\psi_p}(\Mp_4(\Q_p))$, which corresponds to the trivial character of $S_{\phi_p}$. Since
          \begin{align*}
            \epsilon(\tfrac{1}{2},\sigma)=\epsilon(\tfrac{1}{2},\calD_{j+\frac{3}{2}},\psi_\infty)=i^{2j+3+1}=-1,
          \end{align*}
          Gan-Ichino's character $\widetilde{\epsilon}_\phi$ sends $a_1$ to $-1$ and $a_\sigma$ to $+1$. Thus a representation $\pi_{\phi_\infty}^{-,+}\otimes\bigotimes_p \pi_p$ is automorphic. It determines $\Phi_\mathrm{II}$ up to scalar, since $\pi_{\phi_\infty}^{-,+}=\pi_\varphi^{+,+}=\pi_{(k+j-\frac{1}{2},-k+\frac{5}{2})}$. The uniqueness follows from the uniqueness of local components $\pi_{\phi_\infty}^{-,+}$ and $\pi_p$ in the respective local packets $\Pi_{\Phi_v,\psi_v}(\Mp_4(\Q_v))$. Note that the local packets are described explicitly in \cite[C.1.4, C.2.2]{gi20}.
    \item Put $\phi=\sigma\boxtimes S_1 \oplus \sigma'\boxtimes S_1$ and write $S_\phi=(\Z/2\Z)a \oplus (\Z/2\Z)a'$, where $a$ (resp. $a'$) corresponds to $\sigma\boxtimes S_1$ (resp. $\sigma'\boxtimes S_1$). For any prime $p<\infty$, let us write $\pi_p$ for the unique $\psi_p$-unramified representation in $\Pi_{\phi_p,\psi_p}(\Mp_4(\Q_p))$, which corresponds to the trivial character of $S_{\phi_p}$. Since
          \begin{align*}
            \epsilon(\tfrac{1}{2},\sigma)  & =\epsilon(\tfrac{1}{2},\mathcal{D}_{k+j-\frac{3}{2}},\psi_\infty)=(-1)^k, \\
            \epsilon(\tfrac{1}{2},\sigma') & =\epsilon(\tfrac{1}{2},\mathcal{D}_{k-\frac{5}{2}},\psi_\infty)=(-1)^k,
          \end{align*}
          Gan-Ichino's character $\widetilde{\epsilon}_\phi$ sends both $a$ and $a'$ to $(-1)^k$. Thus a representation $\pi_{\phi_\infty}^{+,+}\otimes\bigotimes_p \pi_p$ (resp. $\pi_{\phi_\infty}^{-,-}\otimes\bigotimes_p \pi_p$) is automorphic if $k$ is even (resp. odd). It determines $\Phi_\mathrm{II}$ (resp. $\Phi_\mathrm{III}$) up to scalar, since $\pi_{\phi_\infty}^{+,+}=\pi_{(k+j-\frac{1}{2},-k+\frac{5}{2})}$ and $\pi_{\phi_\infty}^{-,-}=\pi_{(k-\frac{5}{2},-k-j+\frac{1}{2})}$. The uniqueness follows from the uniqueness of local components $\pi_{\phi_\infty}^{\epsilon,\epsilon}$ and $\pi_p$ in the respective local $L$-packets.
  \end{enumerate}
\end{proof}

As a corollary of Lemmas \ref{ap2} and \ref{ap3}, we have the following.
\begin{corollary}
  For any Hecke eigenform $\Phi$ in $\mathcal{A}_{\det^{-k+\frac{5}{2}}\Sym_{2k+j-3}}(\Mp_4)^{E_\psi}$ or $\mathcal{A}_{\det^{-k-j+\frac{1}{2}}\Sym_{2k+j-3}}(\Mp_4)^{E_\psi}$, the automorphic representation $\pi_\Phi$ is cuspidal.
\end{corollary}
\begin{proof}
  If $\phi_\Phi$ is tempered, then the assertion follows from \cite[Proposition 4.1]{gi18}.
  Otherwise, by Lemmas \ref{ap2} and \ref{ap3}, $\phi_\Phi$ is of the form $1\boxtimes S_2 \oplus \sigma\boxtimes S_1$, and $\pi_{\Phi,\infty}=\pi_{\phi_{\Phi,\infty}}^{-,+}$.
  Then the assertion follows from \cite[\S 8.3]{gi20}.
\end{proof}
\subsection{The $A$-parameters associated to automorphic representations of $\SO_5$ generating large discrete series representations at the real place}
Next, we shall consider the automorphic representations of $\SO_5$.
Let $k\geq3$ and $j\geq0$ be integers.
In this subsection, we do not assume any parity condition on $j$.
Let $\Psi \in \mathcal{A}_{\det^{-j-1}\Sym_{2j+2k-2}}(\SO_5)^{\mathrm{unr}}$ be a Hecke eigenform.
We shall write $\pi_\Psi$ for the irreducible automorphic representation of $\SO_5(\A)$ generated by the image of $\Psi$.
Let $\phi_\Psi$ be the $A$-parameter of $\pi_\Psi$.
As in the previous subsection, one can determine the form of the $A$-parameter $\phi_\Psi$ as follows.
We shall put $\varphi=\calD_{k+j-\frac{3}{2}} \oplus \calD_{k-\frac{5}{2}}$ as before.
\begin{lemma}\label{ds2}
  The only local $A$-parameter of which the packet contains $\pi_{\Psi,\infty}$ is $\varphi$, and the character of $S_\varphi$ corresponding to $\pi_{\Psi,\infty}$ is trivial.
\end{lemma}
\begin{proof}
  By the definition of $\mathcal{A}_{\det^{-j-1}\Sym_{2j+2k-2}}(\SO_5)^{\mathrm{unr}}$, we have $\pi_{\Psi,\infty}=\sigma_{(j+k-1,k-2)}$.
  Let $\phi_\infty$ be a local $A$-parameter such that $\sigma_{(j+k-1,k-2)}\in \Pi_{\phi_\infty}(\SO_5(\R))$.
  Then by \cite{mr1, mr2}, $\phi_\infty$ is of the form
  \begin{align*}
     & \calD_{\kappa-\frac{1}{2}} \boxtimes S_1 \oplus \calD_{\eta-\frac{1}{2}}\boxtimes S_1, &
     & \kappa>\eta\geq1,                                                                      &
  \end{align*}
  or
  \begin{align*}
     & \calD_{\kappa-\frac{1}{2}} \boxtimes S_1 \oplus \chi\boxtimes S_2, &
     & \kappa>0,\ \text{$\chi$ is a quadratic character}.                 &
  \end{align*}
  Moreover, by comparing the Blattner parameters, we have
  \begin{align*}
    (j+k-1,k-2)=
    \begin{cases*}
      (\kappa,\eta),     & if $\phi_\infty=\calD_{\kappa-\frac{1}{2}} \boxtimes S_1 \oplus \calD_{\eta-\frac{1}{2}}\boxtimes S_1$ and the corresponding character is trivial,     \\
      (\eta-1,\kappa+1), & if $\phi_\infty=\calD_{\kappa-\frac{1}{2}} \boxtimes S_1 \oplus \calD_{\eta-\frac{1}{2}}\boxtimes S_1$ and the corresponding character is not trivial, \\
      (0,\kappa+1),      & if $\phi_\infty=\calD_{\kappa-\frac{1}{2}} \boxtimes S_1 \oplus \chi\boxtimes S_2$.
    \end{cases*}
  \end{align*}
  Since $j\geq0$ and $k\geq3$, $\phi_\infty$ cannot be of the form $\calD_{\kappa-\frac{1}{2}} \boxtimes S_1 \oplus \chi\boxtimes S_2$.
  Suppose that $(j+k-1,k-2)=(\eta-1, \kappa+1)$.
  Then the condition $\kappa>\eta$ contradicts the assumption $j\geq0$.
  Consequently, we have $(j+k-1,k-2)=(\kappa,\eta)$.
  Thus the local $A$-parameter is
  \begin{align*}
    \phi_\infty
     & =\calD_{\kappa-\frac{1}{2}} \boxtimes S_1 \oplus \calD_{\eta-\frac{1}{2}}\boxtimes S_1 \\
     & =\calD_{j+k-\frac{3}{2}} \oplus \calD_{k-\frac{5}{2}}
    =\varphi,
  \end{align*}
  and the character corresponding to $\pi_{\Psi,\infty}$ is trivial.
\end{proof}

\begin{lemma}\label{api2}
  There exists one and only one of an irreducible everywhere unramified cuspidal symplectic automorphic representation $\tau_\Psi$ of $\GL_4(\A)$ such that $\phi_\Psi=\tau_\Psi\boxtimes S_1$ and
  \begin{align*}
    \begin{aligned}
      L(s,\Psi)
       & =L(s, \tau_\Psi),                                      \\
      \tau_{\Psi,\infty}
       & =\calD_{k+j-\frac{3}{2}} \oplus \calD_{k-\frac{5}{2}},
    \end{aligned}
  \end{align*}
  or a pair $(\sigma_\Psi, \sigma_\Psi')$ of irreducible everywhere unramified cuspidal symplectic automorphic representations $\sigma_\Psi$ and $\sigma_\Psi'$ of $\GL_2(\A)$ such that $\phi_\Psi=\sigma_\Psi \boxtimes S_1 \oplus \sigma_\Psi'\boxtimes S_1$ and
  \begin{gather*}
    L(s,\Psi)
    =L(s, \sigma_\Psi)L(s, \sigma_\Psi'),\\
    \sigma_{\Psi,\infty}
    =\calD_{k+j-\frac{3}{2}}, \qquad
    \sigma_{\Psi,\infty}'
    =\calD_{k-\frac{5}{2}}.
  \end{gather*}
  Moreover, $\tau_\Psi$ or $(\sigma_\Psi,\sigma_\Psi')$ is uniquely determined by $\Psi$.
\end{lemma}
\begin{proof}
  The $A$-parameter $\phi_\Psi$ is one of the following forms:
  \begin{enumerate}[(i)]
    \item $\phi_\Psi=\chi \boxtimes S_4$, where $\chi$ is a quadratic character of $\Q^\times \backslash \A^\times$;
    \item $\phi_\Psi=\chi \boxtimes S_2 \oplus \chi' \boxtimes S_2$, where $\chi$ and $\chi'$ are distinct quadratic characters of $\Q^\times \backslash \A^\times$;
    \item $\phi_\Psi=\sigma \boxtimes S_2$, where $\sigma$ is an irreducible cuspidal orthogonal automorphic representation of $\GL_2(\A)$;
    \item $\phi_\Psi=\chi \boxtimes S_2 \oplus \sigma \boxtimes S_1$, where $\chi$ is a quadratic character of $\Q^\times \backslash \A^\times$, and $\sigma$ is an irreducible cuspidal symplectic automorphic representation of $\GL_2(\A)$;
    \item $\phi_\Psi=\sigma \boxtimes S_1 \oplus \sigma' \boxtimes S_1$, where $\sigma$ and $\sigma'$ are distinct irreducible cuspidal symplectic automorphic representations of $\GL_2(\A)$;
    \item $\phi_\Psi=\tau \boxtimes S_1$, where $\tau$ is an irreducible cuspidal symplectic automorphic representation of $\GL_4(\A)$.
  \end{enumerate}

  Since $\pi_\Psi$ is unramified everywhere, $\chi$, $\chi'$, $\sigma$, $\sigma'$, and $\tau$ are also unramified everywhere.
  Moreover, $\chi$ must be trivial since $\A^\times =\Q^\times \R^\times_{>0} \hat{\Z}^\times$.
  By Lemma \ref{ds2}, the localization $\phi_{\Psi, \infty}$ is equal to $\varphi$.
  Thus the cases (i), (ii), (iii), (iv) cannot occur.
  As Lemma \ref{api}, the uniqueness follows from that of the $A$-parameter.
\end{proof}

Moreover, we can show the existence and uniqueness of a Hecke eigenform $\Psi$ for a given $A$-parameter $\phi$ of the form above:
\begin{lemma}\label{iap2}
  \begin{enumerate}
    \item Let $\tau$ be an irreducible cuspidal symplectic automorphic representation of $\GL_4(\A)$ that is unramified everywhere and satisfies $\tau_\infty=\calD_{k+j-\frac{3}{2}} \oplus \calD_{k-\frac{5}{2}}$.
          Then there exists a Hecke eigenform $\Psi \in \mathcal{A}_{\det^{-j-1}\Sym_{2j+2k-2}}(\SO_5)^{\mathrm{unr}}$ with $\phi_\Psi=\tau\boxtimes S_1$, and it is unique up to a scalar multiple.
    \item Let $(\sigma,\sigma')$ be a pair of irreducible cuspidal symplectic automorphic representations of $\GL_2(\A)$ that are unramified everywhere and satisfy $\sigma_\infty=\mathcal{D}_{k+j-\frac{3}{2}}$ and $\sigma'_\infty=\mathcal{D}_{k-\frac{5}{2}}$.
          Then there exists a Hecke eigenform $\Psi \in \mathcal{A}_{\det^{-j-1}\Sym_{2j+2k-2}}(\SO_5)^{\mathrm{unr}}$ with $\phi_\Psi=\sigma\boxtimes S_1 \oplus \sigma'\boxtimes S_1$, and it is unique up to a scalar multiple.
  \end{enumerate}
\end{lemma}
\begin{proof}
  In the first case, put $\phi=\tau\boxtimes S_1$.
  In the second case, put $\phi=\sigma\boxtimes S_1 \oplus \sigma'\boxtimes S_1$.

  For any prime $p<\infty$, let us write $\pi_p$ for the unique unramified representation in $\Pi_{\phi_p}(\SO_5(\Q_p))$, which corresponds to the trivial character of $S_{\phi_p}$. By the proof of Lemma \ref{ds2}, we have $\sigma_{(j+k-1,k-2)}\in \Pi_{\phi_\infty}(\SO_5(\R))$ and it corresponds to the trivial character of $S_{\phi_\infty}$. Since the parameter $\phi$ is tempered, Arthur's character $\epsilon_\phi$ is trivial. Thus a representation $\sigma_{(j+k-1,k-2)}\otimes\bigotimes_p \pi_p$ is automorphic. This automorphic representation determines $\Psi$ up to scalar. The uniqueness follows from the uniqueness of local components in the local $L$-packets.
\end{proof}

As a corollary of Lemma \ref{api2}, we have the following.
\begin{corollary}
  For any Hecke eigenform $\Psi$ in $\mathcal{A}_{\det^{-j-1}\Sym_{2j+2k-2}}(\SO_5)^{\mathrm{unr}}$, the automorphic representation $\pi_\Psi$ is cuspidal.
\end{corollary}
\begin{proof}
  By the proof of Lemma \ref{api2}, the $A$-parameter $\phi_\Psi$ is tempered.
  Now the assertion follows from the analog of \cite[Proposition 4.1]{gi18} for $\SO_{2n+1}$ (in particular $\SO_5$).
\end{proof}

\subsection{Proofs of the main theorems}\label{Proof}
In this subsection, we finally complete the proofs of the main theorems (Theorem \ref{main}, \ref{lifting}, \ref{compl}, and \ref{main2}).\\

We shall first consider Theorem \ref{main}.
Since the spaces $S_{\det^{j+3}\Sym_{2k-6}}(\Sp_4(\Z))$, $\mathcal{A}_{\det^{-k+\frac{5}{2}}\Sym_{2k+j-3}}(\Mp_4)^{E_\psi}$, and $\mathcal{A}_{\det^{-k-j+\frac{1}{2}}\Sym_{2k+j-3}}(\Mp_4)^{E_\psi}$ have bases consisting of Hecke eigenforms, the assertion follows from Lemmas \ref{api}, \ref{iap}, \ref{ap2}, \ref{ap3}, and \ref{fap}.
Note that the corresponding $A$-parameters are of the form (vi) $\tau\boxtimes S_1$ unless $k=3$, in which case they are of the form (iv) $1\boxtimes S_2 \oplus \sigma\boxtimes S_1$ or (vi) $\tau\boxtimes S_1$.\\

Next, let us give a proof of Theorem \ref{lifting}.
As we have remarked after stating Theorem \ref{lifting}, we may assume that $k$ is greater than 7.
For any Hecke eigenforms $f \in S_{2k-4}(\SL_2(\Z))$ and $g \in S_{2k+2j-2}(\SL_2(\Z))$, let $\tau_f$ and $\tau_g$ be the corresponding automorphic representation of $\GL_2(\A)$, respectively.
Then it is well known that the following properties hold:
\begin{itemize}
  \item $\tau_f$ and $\tau_g$ are irreducible, cuspidal and symplectic;
  \item $\tau_{f, \infty}=\calD_{k-\frac{5}{2}}$ and $\tau_{g,\infty}=\calD_{k+j-\frac{3}{2}}$;
  \item $\tau_{f,p}$ and $\tau_{g,p}$ are unramified for every prime number $p$;
  \item $L(s, \tau_f)=L(s+k-\tfrac{5}{2},f)$ and $L(s, \tau_g)=L(s+j+k-\frac{3}{2},g)$.
\end{itemize}
Therefore, by Lemma \ref{fap}, there exists a Hecke eigenform
\begin{align*}
  \Phi \in \mathcal{A}_{\det^{-k+\frac{5}{2}}\Sym_{2k+j-3}}(\Mp_4)^{E_\psi} \cup \mathcal{A}_{\det^{-k-j+\frac{1}{2}}\Sym_{2k+j-3}}(\Mp_4)^{E_\psi},
\end{align*}
such that
\begin{align*}
  \phi_\Phi & =\tau_g\boxtimes S_1 \oplus \tau_f\boxtimes S_1,
\end{align*}
and hence 
\begin{align*}
  L(s,\Phi) & =L(s,\tau_f)L(s,\tau_g)=L(s+k-\tfrac{5}{2},f)L(s+j+k-\frac{3}{2},g).
\end{align*}
This gives us an injective linear map
\begin{equation*}
  \mathscr{L} : S_{2k-4}(\SL_2(\Z)) \otimes S_{2k+2j-2}(\SL_2(\Z)) \lra S_{\det^{k-\frac{1}{2}}\Sym_j}^+(\Gamma_0(4)),
\end{equation*}
such that if $f \in S_{2k-4}(\SL_2(\Z))$ and $g \in S_{2k+2j-2}(\SL_2(\Z))$ are Hecke eigenforms, then so is $\mathscr{L}(f\otimes g) \in S_{\det^{k-\frac{1}{2}}\Sym_j}^+(\Gamma_0(4))$, and they satisfy
\begin{equation*}
  L(s, \mathscr{L}(f\otimes g)) = L(s+k-\tfrac{5}{2},f)L(s+j+k-\frac{3}{2},g).
\end{equation*}\\

Now we come to Theorem \ref{compl}.
If $(\sigma, \sigma')$ is a pair of irreducible cuspidal symplectic automorphic representations of $\GL_2(\A)$ that are unramified everywhere and satisfy $\sigma_\infty=\calD_{k+j-\frac{3}{2}}$ and $\sigma'_\infty=\calD_{k-\frac{5}{2}}$, then there exist Hecke eigenforms $f \in S_{2k-4}(\SL_2(\Z))$ and $g \in S_{2k+2j-2}(\SL_2(\Z))$ such that the corresponding automorphic representations of $\GL_2(\A)$ are $\sigma'$ and $\sigma$, respectively.
Thus by the proof of Theorem \ref{lifting}, it follows from Lemmas \ref{ap2}, \ref{ap3}, and \ref{fap} that the image of $\mathscr{L}$ is spanned by the Hecke eigenforms $\Phi$ corresponding to the pairs $(\sigma, \sigma')$ of irreducible cuspidal symplectic automorphic representations of $\GL_2(\A)$, and its orthogonal complement $(\myim\mathscr{L})^\perp$ is spanned by those corresponding to irreducible cuspidal symplectic automorphic representations $\tau_F$ of $\GL_4(\A)$.
Combining this with Lemmas \ref{api} and \ref{iap}, we obtain Theorem \ref{compl}.\\

The idea of the proof of Theorem \ref{main2} is the same as that of Theorem \ref{main}.
On the left hand side $\mathcal{A}_{\det^{-j-1}\Sym_{2j+2k-2}}(\SO_5)^{\mathrm{unr}}$ of $\rho'$, we appeal to Lemmas \ref{api2} and \ref{iap2}.
On the right hand side of $\rho'$, we appeal to Lemmas \ref{ap2}, \ref{ap3}, and \ref{fap} if $j$ is odd, and to \cite[Lemmas 6.9 and 6.10]{ishi} if $j$ is even.
The corresponding $A$-parameters are of the form (v) $\sigma\boxtimes S_1 \oplus \sigma'\boxtimes S_1$ or (vi) $\tau\boxtimes S_1$.

\section{Relations with abstract theta lift}\label{relat}
In this section, we shall describe the correspondences in terms of abstract theta liftings.
This is the third main result of this paper.
\subsection{Local theta correspondence}\label{ltc}
We shall first review the theory of theta correspondence between the metaplectic group and the odd special orthogonal group over a local field of general degree.
Let $F$ be a local field of characteristic zero, and $\psi$ a nontrivial additive character of $F$.
Let $(W,\an{-,-})$ be a symplectic space of dimension $2n$, and $(V,(-,-))$ a non-degenerate symmetric bilinear space of discriminant 1 and dimension $2n+1$, over $F$.
We write $\Sp(W)$, $\Mp(W)$, $\Or(V)$, and $\SO(V)$ for the symplectic group of $W$, its metaplectic double cover, the orthogonal group of $V$, and the special orthogonal group of $V$, respectively.
The tensor product $V \otimes _F W$ has a natural symplectic form $\an{\an{-,-}}$ defined by
\begin{align*}
  \an{\an{ v\otimes w, v'\otimes w' }} = (v, v') \cdot \an{ w, w' }.
\end{align*}
Then there is a natural map
\begin{align}\label{natsp}
  \Sp(W) \times \Or(V) \longrightarrow \Sp(V \otimes W).
\end{align}
One has the metaplectic $\C^1$-cover $\mmp(V\otimes W)$ of $\Sp(V \otimes W)$, and the additive character $\psi$ determines the Weil representation $\Omega_\psi$ of $\mmp(V \otimes W)$.
Kudla \cite{spl} gives a splitting of the metaplectic cover over $\Mp(W) \times \Or(V)$, and hence we have a Weil representation $\Omega_{V,W,\psi}$ of $\Mp(W) \times \Or(V)$.

Given an irreducible smooth representation $\sigma$ of $\Or(V)$, the maximal $\sigma$-isotypic quotient of $\Omega_{V,W,\psi}$ is of the form
\begin{align*}
  \sigma \boxtimes \Theta_{V,W,\psi}(\sigma)
\end{align*}
for some smooth representation $\Theta_{V,W,\psi}(\sigma)$ of $\Mp(W)$ (called the big theta lift of $\sigma$).
Then $\Theta_{V,W,\psi}(\sigma)$ is either zero or has finite length.
The maximal semisimple quotient of $\Theta_{V,W,\psi}(\sigma)$ is denoted by $\theta_{V,W,\psi}(\sigma)$ (called the small theta lift of $\sigma$).

Similarly, if $\pi$ is an irreducible genuine representation of $\Mp(W)$, then one has its big theta lift $\Theta_{W,V,\psi}(\pi)$ and its small theta lift $\theta_{W,V,\psi}(\pi)$, which are smooth representations of $\Or(V)$. \\

By the Howe duality, each small theta lift is irreducible or zero (\cite{wal}, \cite{gt1}).
Note that there are finite numbers of symmetric bilinear spaces $V_i$ ($i\in I$) up to isomorphism, and the number $|I|$ is
\begin{align*}
  \begin{cases*}
    2,   & if $F$ is non-archimedean, \\
    n+1, & if $F=\R$,                 \\
    1,   & if $F=\C$.
  \end{cases*}
\end{align*}
Note also that $\Or(V)$ is isomorphic to $\SO(V)\times\{\pm1\}$, and hence any representation of $\SO(V)$ has two extensions to $\Or(V)$.
Adams-Barbasch \cite{ab95, ab98} and Gan-Savin \cite{gs} showed that
\begin{enumerate}
  \item for $\pi \in \Irr(\Mp(W))$, exactly one of $\theta_{W,V_i,\psi}(\pi)$ ($i\in I$) is nonzero;
  \item given $\sigma \in \Irr(\SO(V))$, with extensions $\sigma^\pm$ to $\Or(V)$, exactly one of $\Theta_{V,W,\psi}(\sigma^\pm)$ is nonzero.
\end{enumerate}
These give the following theorem.
\begin{theorem}\label{abgs}
  There is a bijection
  \begin{align*}
    \theta_\psi : \Irr(\Mp(W)) \longleftrightarrow \bigsqcup_{i\in I}\Irr(\SO(V_i)),
  \end{align*}
  given by the theta correspondence with respect to $\psi$.
\end{theorem}
Note that in the paper \cite{gs} the residual characteristic was assumed to be odd since the Howe duality for even residual characteristic was conjecture then.
However, the Howe duality for even residue characteristic was verified by Gan-Takeda \cite{gt1}, so now we have the results of \cite{gs} for arbitrary residue characteristic.

LLC for $\Mp_{2n}(F)=\Mp(W)$ is given by combining Theorem \ref{abgs} with LLC for $\SO_{2n+1}$ and its inner forms.

\subsection{Abstract theta correspondence}\label{atc}
In this subsection let us recall the definition of the abstract theta lifting.

Let $F$ be a number field, and $\psi$ a nontrivial additive character of $F\backslash \A_F$.
As in the last subsection let $W$ and $V$ be symplectic and non-degenerate symmetric bilinear spaces over $F$ of dimension $2n$ and $2n+1$, respectively.
Consider an irreducible genuine representation $\pi=\otimes_v \pi_v$ of $\Mp(W)(\A_F)$.
Assume that the local theta lift $\theta_{\psi_v}(\pi_v)$ is a representation of $\SO(V)(F_v)$ for all places $v$ of $F$.
Then $\theta_{\psi_v}(\pi_v)$ is unramified for almost all places, since local theta lift of an unramified representation is also unramified.
The abstract theta lift $\theta^{\mathrm{abs}}_\psi(\pi)$ of $\pi$ to $\SO(V)(\A_F)$ is defined by
\begin{align*}
  \theta^{\mathrm{abs}}_\psi(\pi)=\bigotimes_v \theta_{\psi_v}(\pi_v),
\end{align*}
which is an irreducible representation of $\SO(V)(\A_F)$.
Conversely, for an irreducible representation $\sigma=\otimes_v \sigma_v$ of $\SO(V)(\A_F)$, an irreducible representation
\begin{align*}
  \theta^{\mathrm{abs}}_\psi(\sigma)=\bigotimes_v \theta_{\psi_v}(\sigma_v)
\end{align*}
of $\Mp_{2n}(W)(\A_F)$ is called the abstract theta lift of $\sigma$.
It is clear that if $\pi=\theta^\mathrm{abs}_\psi(\sigma)$ then $\theta^\mathrm{abs}_\psi(\pi)=\sigma$.
It is also clear that if $\theta^\mathrm{abs}_\psi(\pi)$ exists and denoted by $\sigma$ then $\theta^\mathrm{abs}_\psi(\sigma)=\pi$.

\subsection{Description in terms of abstract theta correspondence}\label{aaaaa}
We shall finally describe Theorems \ref{main} and \ref{compl} in terms of the abstract theta correspondence.
Let $k\geq 3$ be an integer and $j\geq1$ an odd integer.

First, let us consider maps from the space $S_{\det^{j+3}\Sym_{2k-6}}(\Sp_4(\Z))$ of holomorphic cusp forms of degree 2 of integral weight.
The linear maps $\rho^\itN$ (Theorem \ref{main}) and $\rho^\itH$ (Theorem \ref{compl}) give the following two maps preserving $L$-functions:
\begin{itemize}
  \item $F\mapsto\rho_\mathrm{II}(F)$, from a set of linearly independent Hecke eigenforms in $S_{\det^{j+3}\Sym_{2k-6}}(\Sp_4(\Z))$ except Saito-Kurokawa lifts, to a set of those in $\mathcal{A}_{\det^{-k+\frac{5}{2}}\Sym_{2k+j-3}}(\Mp_4)^{E_\psi}$;
  \item $F\mapsto\rho_\mathrm{III}(F)$, from a set of linearly independent Hecke eigenforms in $S_{\det^{j+3}\Sym_{2k-6}}(\Sp_4(\Z))$ except Saito-Kurokawa lifts, to a set of those in $\mathcal{A}_{\det^{-k-j+\frac{1}{2}}\Sym_{2k+j-3}}(\Mp_4)^{E_\psi}$.
\end{itemize}
Then the two maps respectively give assignments $\pi_F \mapsto \pi_{\rho_\mathrm{II}(F)}$ and $\pi_F \mapsto \pi_{\rho_{\mathrm{III}}(F)}$, from some irreducible cuspidal automorphic representations of $\SO_5(\A)$ to irreducible genuine cuspidal automorphic representations of $\Mp_4(\A)$.
\begin{theorem}\label{absth}
  Let $a\in \Q^\times$ be a negative (resp. positive) rational number if $*=\mathrm{II}$ (resp. $\mathrm{III}$).
  We have
  \begin{align}\label{wc}
    \theta^\mathrm{abs}_{\psi_a}(\pi_F \otimes (\chi_a\circ\nu)) & =\pi_{\rho_*(F)},                                                        &
    \pi_F                                                        & =\theta^\mathrm{abs}_{\psi_a}(\pi_{\rho_*(F)}) \otimes (\chi_a\circ\nu),
  \end{align}
  where $\chi_a$ denotes the quadratic (or trivial) character on $\A_\Q^\times$ corresponding to $\Q(\sqrt{a})/\Q$.
  Here note that the character $\chi_a\circ\nu$ of $\GSp_4$ factors through $\PGSp_4\simeq\SO_5$.
\end{theorem}
\begin{proof}
  In the proofs of the main theorems, we have seen that $\pi_F$, $\pi_{\rho_\mathrm{II}(F)}$, and $\pi_{\rho_\mathrm{II}(F)}$ have the same $A$-parameter $\phi$.
  Since we exclude the Saito-Kurokawa lifts, the parameter $\phi$ is tempered.
  We have also seen that the corresponding character on $S_{\phi_p}$ is trivial for every prime $p<\infty$, and that $\pi_{F,\infty}$ and $\pi_{\rho_\mathrm{III}(F),\infty}$ (resp. $\pi_{\rho_\mathrm{II}(F),\infty}$) are associated to the trivial character (resp. $(a,b)\mapsto(-1)^{a+b}$) on $S_{\phi_\infty}\simeq(\Z/2\Z)^2$.

  At all finite prime $p<\infty$, by \cite[Theorem 1.5]{gs} the $L$-parameter of $\theta_{(\psi_a)_p}(\pi_{\rho_*(F),p})$ is $\phi_p\otimes\chi_a$, where $\chi_a$ denotes the character of $L_{\Q_p}$ associated to $\Q_p(\sqrt{a})/\Q_p$.
  In particular $\theta_{(\psi_a)_p}(\pi_{\rho_*(F),p})$ is a representation of $\SO_5(\Q_p)$.
  By \cite[Main Theorem (iv)]{gt2} the $L$-parameter of $\theta^\mathrm{abs}_{\psi_a}(\pi_{\rho_*(F)}) \otimes (\chi_a\circ\nu)$ is $(\phi_p\otimes\chi_a)\otimes\chi_a=\phi_p$.
  The associated $L$-packet $\Pi_{\phi_p}(\SO_5(\Q_p))$ is singleton, and hence
  \begin{align*}
    \theta_{(\psi_a)_p}(\pi_{\rho_*(F),p}) \otimes (\chi_a\circ\nu)_p = \pi_{F,p}.
  \end{align*}
  Let $\varphi$ denote an $L$-parameter $\calD_{k+j-\frac{3}{2}} \oplus \calD_{k-\frac{5}{2}}$ as before.
  At the real place $\infty$, by Lemma \ref{ds} we have
  \begin{align*}
    \pi_{\rho_*(F),\infty}=\pi_\varphi^{-,-},
  \end{align*}
  relative to $(\psi_a)_\infty$, which means that $\theta_{(\psi_a)_\infty}(\pi_{\rho_*(F),\infty})=\pi_{F,\infty}$.
  Note that $\varphi\otimes\chi_a=\varphi$.
  Note also that the packet $\Pi_{\phi_\infty}(\SO_5(\R))$ has two elements: one is generic, and the other is not.
  Since the similitude norm $\nu$ is trivial on every unipotent radical in $\GSp_4$, these imply that tensoring the quadratic character $(\chi_a\circ\nu)_\infty$ change neither the $L$-parameter nor genericity.
  Thus we have
  \begin{align*}
    \theta_{(\psi_a)_\infty}(\pi_{\rho_*(F),\infty}) \otimes (\chi_a\circ\nu)_\infty = \pi_{F,\infty}.
  \end{align*}
  Now the assertion follows.
\end{proof}

Next, let us consider maps to the space $\mathcal{A}_{\det^{-j-1}\Sym_{2j+2k-2}}(\SO_5)^{\mathrm{unr}}$ of cuspidal automorphic forms on $\SO_5(\A)$ generating large discrete series representations at the real place.
The linear maps $(\rho')\inv$ (Theorem \ref{main2}) and $(\rho')\inv \circ \rho^\itH \circ (\rho^\itN)\inv$ (Theorems \ref{main2}, \ref{main}, and \ref{compl}) give the following two maps preserving $L$-functions:
\begin{itemize}
  \item $\Phi\mapsto\rho'_\mathrm{II}(\Phi)$, from a set of linearly independent Hecke eigenforms in $\mathcal{A}_{\det^{-k+\frac{5}{2}}\Sym_{2k+j-3}}(\Mp_4)^{E_\psi}$ except Saito-Kurokawa lifts, to a set of those in $\mathcal{A}_{\det^{-j-1}\Sym_{2j+2k-2}}(\SO_5)^{\mathrm{unr}}$;
  \item $\Phi\mapsto\rho'_\mathrm{III}(\Phi)$, from a set of linearly independent Hecke eigenforms in $\mathcal{A}_{\det^{-k-j+\frac{1}{2}}\Sym_{2k+j-3}}(\Mp_4)^{E_\psi}$ to a set of those in $\mathcal{A}_{\det^{-j-1}\Sym_{2j+2k-2}}(\SO_5)^{\mathrm{unr}}$.
\end{itemize}
Note that the directions of maps are opposite to those of $\rho_\mathrm{II}$ and $\rho_\mathrm{III}$.
Then the two maps respectively give assignments $\pi_\Phi \mapsto \pi_{\rho'_\mathrm{II}}$ and $\pi_\Phi \mapsto \pi_{\rho'_\mathrm{II}}$, from some irreducible genuine cuspidal automorphic representations of $\Mp_4(\A)$ to irreducible cuspidal automorphic representations of $\SO_5(\A)$.
\begin{theorem}\label{absth2}
  Let $a\in \Q^\times$ be a positive (resp. negative) rational number if $*=\mathrm{II}$ (resp. $\mathrm{III}$).
  We have
  \begin{align}\label{wc2}
    \theta^\mathrm{abs}_{\psi_a}(\pi_\Phi) \otimes (\chi_a\circ\nu) & =\pi_{\rho'_*(\Phi)},                                                        &
    \pi_\Phi                                                        & =\theta^\mathrm{abs}_{\psi_a}(\pi_{\rho'_*(\Phi)} \otimes (\chi_a\circ\nu)).
  \end{align}
\end{theorem}
\begin{proof}
  The proof is similar to that of Theorem \ref{absth}.
\end{proof}


\end{document}